\documentclass[11pt]{amsart}
\PassOptionsToPackage{dvipsnames}{xcolor}
\usepackage[margin=1in]{geometry}
\usepackage{amsmath,amssymb,mathtools,tikz}
\usetikzlibrary{decorations.pathreplacing}
\usepackage{microtype,nicefrac}
\usepackage[dvipsnames]{xcolor}
\usepackage[colorlinks=true,linkcolor=RoyalBlue,citecolor=PineGreen,urlcolor=BrickRed]{hyperref}
\definecolor{oiblue}{HTML}{0072B2}
\definecolor{oisky}{HTML}{56B4E9}
\definecolor{oiorange}{HTML}{E69F00}
\definecolor{oivermillion}{HTML}{D55E00}
\colorlet{figchain}{oiblue!70!black}
\colorlet{figtube}{oisky!30}
\colorlet{figshade}{oisky!18}
\colorlet{figcube}{oiblue}
\colorlet{figguide}{gray!65}
\colorlet{figfaint}{gray!35}
\colorlet{fighigh}{oivermillion}
\colorlet{fighit}{oiorange!45}
\newcommand{\nf}{\nicefrac}
\newcommand{\E}{\mathbb E}
\newcommand{\Pp}{\mathbb P}
\newcommand{\Var}{\operatorname{Var}}
\newcommand{\Os}{\mathcal O}
\newcommand{\repourl}{\url{https://github.com/nitromannitol/FPP-variance}}
\newcommand{\R}{\mathbb R}
\newcommand{\Z}{\mathbb Z}
\newcommand{\dist}{\operatorname{dist}}
\newtheorem{theorem}{Theorem}[section]
\numberwithin{equation}{section}
\newtheorem{lemma}[theorem]{Lemma}
\newtheorem{proposition}[theorem]{Proposition}
\newtheorem{assumption}{Assumption}

\theoremstyle{remark}

\title[Fluctuation bounds for invariant FPP]{Fluctuation bounds for rotationally invariant first-passage percolation}
\author{Scott Armstrong}
\author{Ahmed Bou-Rabee}
\date{}
\begin{document}
\begin{abstract}
We prove that in rotationally invariant first-passage percolation on~$\R^d$ with~$d\geq2$, the passage time~$T_n$ over a distance~$n$ satisfies~$\E[\exp(|T_n-\mu n|/n^{\frac12(1-\varepsilon)})]\leq C$ for some~$\varepsilon>0$, where~$\mu$ is the time constant. This implies that its variance is at most a constant times~$n^{1-\varepsilon}$, the bound that Basu, Sidoravicius and Sly~\cite{BSS23,BSS} proved in the planar case. As in their work, our argument is an induction on the scale, but in place of their proof that the geodesic is chaotic, we use the averaging argument of Benjamini, Kalai and Schramm~\cite{BKS} together with a logarithmic improvement of the Efron--Stein inequality due to Bena\"im and Rossignol~\cite{BR}.
\end{abstract}
\maketitle

\section{Introduction}

In continuum first-passage percolation, one measures distances in~$\R^d$ by integrating a random density along curves. A natural example starts from Gaussian white noise~$\omega$ on~$\R^d$ with~$d\geq2$. Convolving it with a smooth radial function~$\chi$ supported in the closed unit ball, and applying a continuous function~$\psi$ with values between two positive constants, gives the random density
\begin{equation*}
\rho\coloneqq\psi(\chi*\omega)\,.
\end{equation*}
Thus the value of~$\rho$ at a point depends only on the noise within unit distance of that point, and~$\rho$ is equivariant under rigid motions. The \emph{passage time} between two points~$u,v\in\R^d$ is
\begin{equation}\label{eq:passage}
T(u,v)\coloneqq\inf_{\pi:u\to v}\int_\pi\rho(x)\,|dx|\,,
\end{equation}
where the infimum is over rectifiable curves~$\pi$ from~$u$ to~$v$. By subadditivity, the \emph{time constant}~$\mu\coloneqq\lim_{t\to\infty}t^{-1}\E[T(0,te_1)]$ exists, and it is positive and finite since~$\rho$ lies between two positive constants.

\begin{theorem}\label{thm:main}
For the density~$\rho=\psi(\chi*\omega)$, there exist~$\varepsilon>0$ and~$C<\infty$ such that, for every~$x\in\R^d$ with~$|x|\geq1$,
\begin{equation*}
\E\Bigl[\exp\Bigl(\frac{|T(0,x)-\mu|x||}{C|x|^{\frac12(1-\varepsilon)}}\Bigr)\Bigr]\leq2\,.
\end{equation*}
\end{theorem}

Theorem~\ref{thm:main} implies that~$\Var[T(0,x)]\leq2C^2|x|^{1-\varepsilon}$ for~$|x|\geq1$. It also improves the classical concentration on the scale~$|x|^{\nf12}$, due on the lattice to Kesten~\cite[(1.15)]{Kesten} and, with Gaussian tails, to Talagrand~\cite[(8.3.2)]{Talagrand95}. In first-passage percolation on~$\Z^d$, with independent and identically distributed edge weights, Kesten proved that~$\Var[T(0,x)]\leq C|x|$ under mild conditions, and asked for a bound by~$|x|^{1-\varepsilon}$~\cite[Remark~1]{Kesten}. Benjamini, Kalai and Schramm~\cite{BKS} improved his bound by a factor~$1/\log|x|$ for weights taking two values with equal probability, and this has since been extended to broad classes of weights~\cite{BR,DHS}. Kesten's question remains open~\cite[Question~13]{ADH}. However, in a companion paper~\cite{ABlattice}, we answer it positively in every direction in which the time constant satisfies a nondegenerate~$C^{2,\alpha}$ condition. No such condition is known for any distribution of the weights; however, for exponential weights on~$\Z^2$, strict convexity and differentiability of the limit shape were recently announced~\cite{OpenAI26}.

For~$d=2$, with~$\chi\geq0$ and~$\psi$ smooth and increasing, this variance bound was proved by Basu, Sidoravicius and Sly~\cite{BSS23,BSS}. Their proof is an induction on scales. Write~$T_n\coloneqq T(0,ne_1)$. Since~$T_{Kn}$, for a large integer~$K$, can be compared with a sum of~$K$ independent passage times across slabs of width~$n$, the Efron--Stein inequality gives~$\Var[T_{Kn}]\leq CK\Var[T_n]$. Replacing~$C$ by~$\nf12$ and iterating over the scales~$K^j$ gives a power saving. To show that~$C$ can be replaced by~$\nf12$, Basu, Sidoravicius and Sly prove that the geodesic is chaotic, in the sense of Chatterjee~\cite{Chatterjee08}: after the noise is resampled in a small random fraction of space, the new geodesic has a small overlap with the old one. This carries out a suggestion of Kesten, who noted that a power might be gained by bounding the probability that an edge lies on the geodesic. Their proof uses the FKG inequality, which holds under their assumptions~$\chi\geq0$ and~$\psi$ increasing. It also uses estimates from~\cite{BSS23}, which rely on planarity to bound the variance from below.

We follow a different multiscale strategy: we induct not on the variance but on an exponential moment of~$|T_n-\mu n|$ on an explicit scale~$n^{\frac12(1-\varepsilon)}$. Exponential tails give bounds that hold simultaneously for many pairs of points, and the induction controls the mean as well, so we need neither the estimates of~\cite{BSS23} nor a lower bound on the variance. We also obtain the improvement on the Efron--Stein inequality differently. As Benjamini, Kalai and Schramm did on the lattice, we average the passage time over translations of its endpoints. Rotational invariance makes the limit shape a ball, so the fluctuation bounds of the induction keep near-geodesics close to the segment between their endpoints. Hence each piece of the noise affects only a small fraction of the translated passage times, and the inequality of Bena\"im and Rossignol~\cite{BR} gives a gain logarithmic in~$K$. Our proof uses neither planarity nor the FKG inequality, so it applies in every dimension and to every~$\chi$ and~$\psi$ as above.

\subsection{General densities}\label{ssec:general}

\begin{assumption}\label{ass:local}
Let~$d\geq2$, and let~$\omega$ be Gaussian white noise or a homogeneous Poisson process on~$\R^d$. The density~$\rho(\omega,x)$ is jointly Borel measurable, and:
\begin{enumerate}
\item \emph{Locality:} $\rho(\omega,x)=\rho(\omega',x)$ whenever~$\omega$ and~$\omega'$ agree on the closed unit ball centered at~$x$;
\item \emph{Boundedness:} $0<\inf_{\omega,x}\rho(\omega,x)\leq\sup_{\omega,x}\rho(\omega,x)<\infty$;
\item \emph{Equivariance under rigid motions:} $\rho(\omega\circ g^{-1},gx)=\rho(\omega,x)$ for every composition~$g$ of a rotation and a translation.
\end{enumerate}
\end{assumption}

The passage time~$T$ and the time constant~$\mu$ are defined as above.

\begin{theorem}\label{thm:general}
Under Assumption~\ref{ass:local}, there exist~$\varepsilon>0$ and~$C<\infty$ such that, for every~$x\in\R^d$ with~$|x|\geq1$,
\begin{equation*}
\E\Bigl[\exp\Bigl(\frac{|T(0,x)-\mu|x||}{C|x|^{\frac12(1-\varepsilon)}}\Bigr)\Bigr]\leq2\,.
\end{equation*}
\end{theorem}

Since~$\psi(\chi*\omega)$ satisfies Assumption~\ref{ass:local} also when~$\omega$ is a Poisson process, Theorem~\ref{thm:general} contains Theorem~\ref{thm:main} and the Poisson case of the theorem of Basu, Sidoravicius and Sly.

\subsection{Proof overview}\label{ssec:overview}

By stationarity and rotational invariance,~$T(u,v)$ has the law of~$T_t$ with~$t=|v-u|$. Replacing~$\rho$ by~$\rho/\mu$, we assume that~$\mu=1$, so that~$\E[T_t]\geq t$ for every~$t>0$, by subadditivity. Fix a large integer~$K$ and a large constant~$A$, and let~$Q_s\coloneqq As^{\nf12-\nf1K}$. We induct on the following statement, for~$s\geq1$:
\begin{equation}\label{eq:hypothesis}
\E\bigl[\exp\bigl(|T_s-s|/Q_s\bigr)\bigr]\leq2\,.
\end{equation}
It controls both the \emph{random fluctuations}~$T_s-\E[T_s]$ and the \emph{nonrandom fluctuations}~$\E[T_s]-s$. For small~$s$, it holds if~$A$ is large, since~$\rho$ is bounded; this is the base case of the induction.

The induction step derives~\eqref{eq:hypothesis} at a distance~$Kn$ from~\eqref{eq:hypothesis} at the distances~$s\leq Kn/2$. By union bounds, the latter give an exponential bound at the distance~$Kn$, but with a large constant:
\begin{equation*}\tag{a}
\E\bigl[\exp\bigl(8|T_{Kn}-\E[T_{Kn}]|/Q_{Kn}\bigr)\bigr]\leq B\,,
\end{equation*}
where~$B$ does not depend on~$A$ or~$K$. This bound cannot be iterated, since the constant would grow at each step. Our main estimate improves on the bound~$\Var[T_{Kn}]\leq CKQ_n^2=CK^{\nf2K}Q_{Kn}^2$ given by the Efron--Stein inequality:
\begin{equation*}\tag{b}
\Var[T_{Kn}]\leq\delta_KQ_{Kn}^2\,,\qquad\text{with~$\delta_K\to0$ as~$K\to\infty$.}
\end{equation*}
We interpolate between~(a) and~(b) with the Cauchy--Schwarz inequality:
\begin{equation*}
\E\bigl[\exp\bigl(4|T_{Kn}-\E[T_{Kn}]|/Q_{Kn}\bigr)\bigr]\leq1+4(B\delta_K)^{\nf12}\,.
\end{equation*}
By subadditivity,~$\E[T_{Kn}]\geq Kn$, and a separate argument gives
\begin{equation*}\tag{c}
\E[T_{Kn}]-Kn\leq Q_{Kn}/25\,.
\end{equation*}
For~$K$ large, the last two bounds give~$\E[\exp(2|T_{Kn}-Kn|/Q_{Kn})]\leq2$, which is~\eqref{eq:hypothesis} at~$s=Kn$ with~$Q_{Kn}/2$ in place of~$Q_{Kn}$.

To prove~(a), (b) and~(c), we compare~$T_{Kn}$ with independent passage times across the~$K$ slabs~${[(i-1)n,in]\times\R^{d-1}}$ (Figure~\ref{fig:slabs}). Here a \emph{slab} of width~$s$ is the image of~$[0,s]\times\R^{d-1}$ under a rigid motion, and its \emph{faces} are the images of~$\{0\}\times\R^{d-1}$ and~$\{s\}\times\R^{d-1}$. For a closed set~$P\subseteq\R^d$, let~$\rho_P$ be~$\rho$ at distance more than one from the boundary of~$P$ and~$\sup\rho$ elsewhere in~$P$, and, for~$x,y\in P$, let
\begin{equation}\label{eq:slabtime}
T_P(x,y)\coloneqq\inf_\pi\int_\pi\rho_P(z)\,|dz|\,,
\end{equation}
where the infimum is over the curves~$\pi$ from~$x$ to~$y$ in~$P$. We write~$T^{(i)}$ for~$T_P$ when~$P$ is the~$i$th slab. By the locality in Assumption~\ref{ass:local}, $T_P$ is determined by~$\omega|_P$, so the~$T^{(i)}$ are independent, and, by~\eqref{eq:hypothesis}, each fluctuates on the scale~$Q_n$.

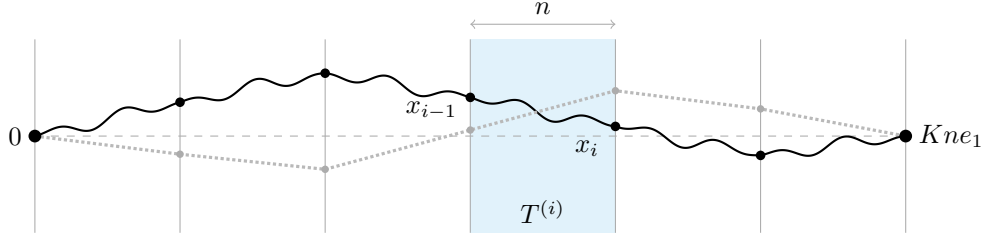
\begin{figure}[t]
\centering
\begin{tikzpicture}[x=.8cm,y=.8cm,every node/.style={font=\small}]
\def\h{0.4}
\fill[figshade] (7.2,-1.6) rectangle (9.6,1.6);
\draw[figguide,dashed] (0,0)--(14.4,0);
\foreach \i in {0,...,6}{\draw[figguide] (2.4*\i,-1.6)--(2.4*\i,1.6);}
\foreach \i/\ya/\yb/\ph in {0/0/1.4/0,1/1.4/2.6/70,2/2.6/1.6/150,3/1.6/0.4/220,4/0.4/-0.8/300,5/-0.8/0/30}{
  \draw[thick] plot[domain=0:1,samples=90,smooth,variable=\t]
    ({2.4*\i+2.4*\t},{\h*(\ya+(\yb-\ya)*\t)+0.13*sin(180*\t)*sin(720*\t+\ph)});
}
\draw[gray!55,densely dotted,very thick] (0,0)--(2.4,-.3)--(4.8,-.55)--(7.2,.1)--(9.6,.75)--(12,.45)--(14.4,0);
\foreach \x/\y in {2.4/-.3,4.8/-.55,7.2/.1,9.6/.75,12/.45}{\fill[figguide] (\x,\y) circle (1.4pt);}
\foreach \i/\y in {1/1.4,2/2.6,3/1.6,4/0.4,5/-0.8}{\fill (2.4*\i,\h*\y) circle (1.8pt);}
\fill (0,0) circle (2.4pt) (14.4,0) circle (2.4pt);
\node[left] at (-.05,0) {$0$};
\node[right] at (14.45,0) {$Kne_1$};
\draw[<->,figguide] (7.2,1.85)--(9.6,1.85);
\node[above] at (8.4,1.85) {$n$};
\node[left] at (7.15,.42) {$x_{i-1}$};
\node[below left] at (9.5,.1) {$x_i$};
\node at (8.4,-1.25) {$T^{(i)}$};
\end{tikzpicture}
\caption{Slabs of width~$n$ between~$0$ and~$Kne_1$: the black dots form a minimizing chain, joined by near-geodesics for the slab passage times~$T^{(i)}$ (solid), and the grey dots joined by the dotted line form another chain.}\label{fig:slabs}
\end{figure}

A \emph{near-geodesic} from~$x$ to~$y$ in~$P$ is a rectifiable curve~$\pi$ from~$x$ to~$y$ in~$P$ with~$\int_\pi\rho_P(z)\,|dz|$ at most~$T_P(x,y)+1$. We write~$\Pi_{T_P}(x,y)$ for the set of these curves, and omit ``in~$P$'' when~$P=\R^d$. Modifying the near-geodesics from~$0$ to~$Kne_1$ near the faces shows that
\begin{equation}\label{eq:F}
F\coloneqq\min\sum_{i=1}^KT^{(i)}(x_{i-1},x_i)\qquad\text{satisfies}\qquad\E\Bigl[\exp\Bigl(\frac{|F-T_{Kn}|}{CK\log^2(2+Kn)}\Bigr)\Bigr]\leq2\,,
\end{equation}
where the minimum is over the \emph{chains}~$0=x_0,x_1,\ldots,x_K=Kne_1$ with~$x_i\cdot e_1=in$ and~$|x_i-x_{i-1}-ne_1|\leq n/4$. The error~$K\log^2(2+Kn)$ is much smaller than~$Q_{Kn}$. A \emph{minimizing chain} is one that attains the minimum.

\emph{Step 1 (Section~\ref{sec:exponential}).} We prove~(a) and the following uniform version of~\eqref{eq:hypothesis}. Fix~$s\in[n,Kn]$ and two cubes of side~$(sQ_s)^{\nf12}$ in the faces~$\{x\cdot e_1=0\}$ and~$\{x\cdot e_1=s\}$, whose centers differ by~$se_1$ plus a vector of length at most~$s$ orthogonal to~$e_1$ (Figure~\ref{fig:stepone}). Then
\begin{equation*}
\E\Bigl[\exp\Bigl(\sup_{u,v}\bigl|T(u,v)-|v-u|\bigr|/(CQ_s)\Bigr)\Bigr]\leq2\,,
\end{equation*}
where~$u$ and~$v$ range over the two cubes and~$C=C(d,\rho)$. We need uniformity because the points of a minimizing chain are random. The side~$(sQ_s)^{\nf12}$ is the natural scale for moves orthogonal to~$e_1$: moving an endpoint of~$[0,se_1]$ by~$\lambda(sQ_s)^{\nf12}$ orthogonally to~$e_1$ lengthens the segment by about~$\lambda^2Q_s/2$. To prove the bound for~$s\leq Kn/4$, we combine~\eqref{eq:hypothesis} at dyadic scales by a chaining argument. For larger~$s$, we cut the segment of length~$s$ into a fixed number of slabs and take a union bound over the sequences of cubes visited by a chain. This also gives~(a). The same union bound shows that~$|x_i-x_{i-1}-ne_1|$ is at most of order~$(KnQ_n)^{\nf12}$ for a minimizing chain of~$F$.

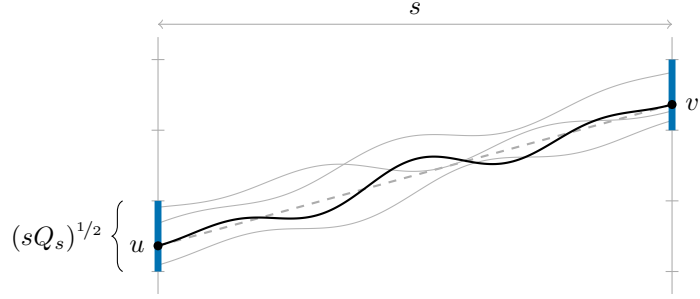
\begin{figure}[t]
\centering
\begin{tikzpicture}[x=.85cm,y=.85cm,every node/.style={font=\small}]
\foreach \x in {0,8}{
  \draw[figguide] (\x,-.9)--(\x,3.1);
  \foreach \k in {0,...,3}{\draw[figguide] (\x-.1,-.55+1.1*\k)--(\x+.1,-.55+1.1*\k);}
}
\draw[figguide,dashed,thick] (0,-.15)--(8,2.05);
\foreach \a/\b/\p in {-.45/1.8/0,.2/2.55/120,.45/1.95/240}{
  \draw[gray!60] plot[domain=0:1,samples=90,smooth,variable=\t]
    ({8*\t},{\a+(\b-\a)*\t+.22*sin(180*\t)*sin(720*\t+\p)});
}
\draw[thick] plot[domain=0:1,samples=90,smooth,variable=\t]
  ({8*\t},{-.15+2.2*\t+.25*sin(180*\t)*sin(900*\t)});
\draw[figcube,line width=2.6pt] (0,-.55)--(0,.55) (8,1.65)--(8,2.75);
\fill (0,-.15) circle (1.7pt) (8,2.05) circle (1.7pt);
\node[left] at (-.05,-.15) {$u$};
\node[right] at (8.05,2.05) {$v$};
\draw[decorate,decoration={brace,amplitude=4pt,mirror}] (-.55,.55)--(-.55,-.55);
\node[left] at (-.7,0) {$(sQ_s)^{\nf12}$};
\draw[<->,figguide] (0,3.3)--(8,3.3);
\node[above] at (4,3.3) {$s$};
\end{tikzpicture}
\caption{One slab of width~$s$, its faces cut by tick marks into cubes of side~$(sQ_s)^{\nf12}$: the thick blue segments are two of these cubes, the thick black curve is a near-geodesic from~$u$ to~$v$ with the dashed segment~$[u,v]$, and the grey curves are near-geodesics between other points of the two cubes.}\label{fig:stepone}
\end{figure}

\emph{Step 2 (Section~\ref{sec:localization}).} We prove that near-geodesics stay close to the segment between their endpoints: for two cubes as in Step~1 with~$s=n$, and~$S$ the slab between their faces,
\begin{equation}\label{eq:transverse}
\E\Bigl[\exp\Bigl(\sup_{u,v,\pi,x}\dist(x,[u,v])^2/(CnQ_n)\Bigr)\Bigr]\leq2\,,
\end{equation}
where~$u$ and~$v$ range over the two cubes, $\pi$ over the near-geodesics from~$u$ to~$v$ in~$S$, and~$x$ over the points of~$\pi$. We also prove that, for every~$\eta>0$ and some~$h>0$ depending on~$\eta$,
\begin{equation*}
\E\Bigl[\sup\bigl|T_S(u,v)-|v-u|-T_S(\widetilde u,\widetilde v)+|\widetilde v-\widetilde u|\bigr|^2\Bigr]^{\nf12}\leq\eta Q_n\,,
\end{equation*}
where the supremum is over the points~$u,\widetilde u$ of one cube and~$v,\widetilde v$ of the other with~$|u-\widetilde u|$ and~$|v-\widetilde v|$ at most~$h(nQ_n)^{\nf12}$. For~\eqref{eq:transverse}, a curve that passes at distance~$\lambda(nQ_n)^{\nf12}$ from the segment is longer by order~$\lambda^2Q_n$, more than Step~1 allows when~$\lambda$ is large. To prove the continuity in the endpoints, we reroute the near-geodesics near their endpoints.

With~$n^{\nf12}$ in place of~$Q_n$, transverse bounds like~\eqref{eq:transverse} go back to Newman and Piza~\cite[Theorem~6]{NP95}, in directions in which the limit shape of a lattice model is curved, and to Howard and Newman~\cite[Theorem~2.4]{HN} for Euclidean models, where isotropy makes the limit shape a ball; both obtain transverse fluctuations of order~$n^{\nf34}$, up to a factor~$n^\varepsilon$. Alexander~\cite[Theorems~1.3 and~1.6]{Alexander20} proved uniform versions for isotropic random-graph models with exponential tails at all distances. In~\eqref{eq:transverse}, the scale~$Q_n$ comes instead from~\eqref{eq:hypothesis} at smaller distances.

\emph{Step 3 (Section~\ref{sec:gain}).} We prove~(b) with the following modified Poincar\'e inequality.

\begin{proposition}[{Bena\"im and Rossignol~\cite[Proposition~2.1]{BR}}]\label{prop:BR}
Let~$\alpha\in(0,1]$, and let~$X_1,X_2,\ldots$ be independent random variables, each taking finitely many values, each with probability at least~$\alpha$. Then every square-integrable function~$f$ of~$(X_j)_j$ satisfies
\begin{equation*}
\Var[f]\log\frac{\Var[f]}{\sum_j\E\bigl[|f-\E_j[f]|\bigr]^2}\leq\frac4\alpha\sum_j\E\bigl[(f-\E_j[f])^2\bigr]\,,
\end{equation*}
where~$\E_j[f]$ denotes the expectation of~$f$ with respect to~$X_j$ alone.
\end{proposition}

Bena\"im and Rossignol assume the logarithmic Sobolev inequality~$\operatorname{Ent}[g^2]\leq(4/\alpha)\Var[g]$ for the functions~$g$ of each~$X_j$, where~$\operatorname{Ent}[g^2]\coloneqq\E[g^2\log g^2]-\E[g^2]\log\E[g^2]$. It holds here, since~$g^2\leq\E[g^2]/\alpha$ gives~$\operatorname{Ent}[g^2]\leq\Var[g^2]/\E[g^2]\leq4\Var[g]/\alpha$.

Let~$L$ be a large integer, and let~$G$ range over the \emph{blocks}
\begin{equation*}
[(i-1)n,in]\times(nQ_n)^{\nf12}\bigl(z+[0,L)^{d-1}\bigr)\,,\qquad1\leq i\leq K,\ z\in L\Z^{d-1}\,.
\end{equation*}
We apply Proposition~\ref{prop:BR} with variables~$X_G$ that approximate the normalized passage times~$(T^{(i)}(u,v)-n)/Q_n$ across~$G$. Each~$X_G$ is a function of~$\omega|_G$ with finitely many values, each of probability at least a constant~$\alpha_L>0$ independent of~$n$, and~$\E_G$ is the expectation with respect to~$X_G$ alone. By Step~2, changing~$\omega|_G$ changes~$F$ only if, before or after the change, a segment~$[x_{i-1},x_i]$ of a minimizing chain passes within distance of order~$(nQ_n)^{\nf12}$ of~$G$. Proposition~\ref{prop:BR} is useful only if~$\sum_G\E[|f-\E_G[f]|]^2$ is much smaller than~$\sum_G\E[(f-\E_G[f])^2]$. We cannot prove this for~$F$: the blocks containing~$0$ and~$Kne_1$ always matter, and elsewhere we would need a lower bound on transverse fluctuations.

Following the averaging trick of Benjamini, Kalai and Schramm~\cite[Lemma~3]{BKS}, we therefore apply Proposition~\ref{prop:BR} to the average of~$F$ over~$\lceil K^{\nf23}\rceil$ translates, whose slabs and endpoints are shifted by~$mL(nQ_n)^{\nf12}e_2$ for~$0\leq m<\lceil K^{\nf23}\rceil$ (see also~\cite[Lemma~5.6]{BR} and~\cite[Section~3.1]{DHS}). By Steps~1 and~2, for each block~$G$, the near-geodesics joining the points of a minimizing chain of~$F$ meet on average at most order~$K^{\nf12}$ of the blocks~$G+mL(nQ_n)^{\nf12}e_2$ with~$m\in\Z$ (Figure~\ref{fig:averaging}, left), so, by stationarity, each block is met by the near-geodesics of at most order~$K^{\nf12}$ of the translates on average (Figure~\ref{fig:averaging}, right). Let~$\overline F$ be this average and~$f\coloneqq\E[\overline F\mid(X_G)_G]$; by Steps~1 and~2, the conditional variance~$\E[\Var[\overline F\mid(X_G)_G]]$ is small. Resampling~$\omega|_G$ changes~$f$ only through the translates whose near-geodesics meet~$G$, a fraction of order~$K^{\nf12}/K^{\nf23}=K^{-\nf16}$ of them, so the two sums in Proposition~\ref{prop:BR} satisfy
\begin{equation*}
\sum_G\E\bigl[(f-\E_G[f])^2\bigr]\leq CKQ_n^2\qquad\text{and}\qquad\sum_G\E\bigl[|f-\E_G[f]|\bigr]^2\leq CK^{-\nf16}KQ_n^2\,.
\end{equation*}
Proposition~\ref{prop:BR} then gives~$\Var[f]\leq CKQ_n^2/(\alpha\log K)$.

Since the shifts are much smaller than~$(KnQ_{Kn})^{\nf12}$, the continuity in the endpoints of Step~2, at the distance~$Kn$, and~\eqref{eq:F} show that~$\overline F$ differs from~$T_{Kn}$ by a small multiple of~$Q_{Kn}$ in mean square, which gives~(b).

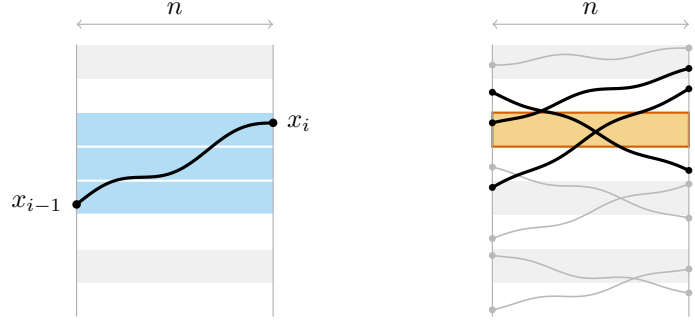
\begin{figure}[t]
\centering
\begin{tikzpicture}[x=1cm,y=.9cm,every node/.style={font=\small}]
\begin{scope}
\foreach \k in {1,3,5,7}{\fill[gray!12] (0,.5*\k) rectangle (2.6,.5*\k+.5);}
\foreach \k in {3,4,5}{\fill[oisky!45] (0,.5*\k) rectangle (2.6,.5*\k+.5);}
\foreach \k in {2,...,5}{\draw[white,line width=.8pt] (0,.5*\k)--(2.6,.5*\k);}
\foreach \x in {0,2.6}{\draw[figguide] (\x,0)--(\x,4);}
\draw[very thick] plot[domain=0:1,samples=80,smooth,variable=\t] ({2.6*\t},{1.65+1.2*\t+.12*sin(540*\t)});
\fill (0,1.65) circle (1.6pt) (2.6,2.85) circle (1.6pt);
\node[left] at (-.05,1.65) {$x_{i-1}$};
\node[right] at (2.65,2.85) {$x_i$};
\draw[<->,figguide] (0,4.3)--(2.6,4.3); \node[above] at (1.3,4.3) {$n$};
\end{scope}
\begin{scope}[xshift=5.5cm]
\foreach \k in {1,3,5,7}{\fill[gray!12] (0,.5*\k) rectangle (2.6,.5*\k+.5);}
\fill[fighit] (0,2.5) rectangle (2.6,3); \draw[fighigh,thick] (0,2.5) rectangle (2.6,3);
\foreach \x in {0,2.6}{\draw[figguide] (\x,0)--(\x,4);}
\foreach \ya/\yb/\c/\w/\p in {.1/.7/gray!55/.6pt/0,.9/.35/gray!55/.6pt/40,1.15/1.95/gray!55/.6pt/80,2.2/1.45/gray!55/.6pt/160,3.7/3.95/gray!55/.6pt/280,1.9/3.35/black/1.2pt/120,2.85/3.65/black/1.2pt/200,3.3/2.15/black/1.2pt/240}{
  \draw[\c,line width=\w] plot[domain=0:1,samples=60,smooth,variable=\t] ({2.6*\t},{\ya+(\yb-\ya)*\t+.08*sin(180*\t)*sin(540*\t+\p)});
  \fill[\c] (0,\ya) circle (1.3pt) (2.6,\yb) circle (1.3pt);
}
\draw[<->,figguide] (0,4.3)--(2.6,4.3); \node[above] at (1.3,4.3) {$n$};
\end{scope}
\end{tikzpicture}
\caption{One slab of width~$n$ cut into blocks (grey and white bands): on the left, a near-geodesic from~$x_{i-1}$ to~$x_i$ crosses it and meets the three blue blocks; on the right, near-geodesics of~$\lceil K^{\nf23}\rceil=8$ translates of~$F$, with~$0$ and~$Kne_1$ shifted by multiples of the block side, cross it, and the three black ones meet the orange block.}\label{fig:averaging}
\end{figure}

\emph{Step 4 (Section~\ref{sec:completion}).} We prove~(c) by showing that~$\E[T_s]-s\leq Q_s/100$ for~$s=Kn/4$, which gives~(c) by subadditivity. We place~$\lceil s\rceil$ slabs of width~$s$ in a row, let~$F$ be the minimum~\eqref{eq:F} for this row, and bound~$\E[F]-s\lceil s\rceil$ from both sides:
\begin{equation*}
\lceil s\rceil\Bigl(\frac{\E[T_s]-s}{1+\sup\rho}-C\kappa Q_s\Bigr)\leq\E[F]-s\lceil s\rceil\leq C\lceil s\rceil\log^2(2+s)\,.
\end{equation*}
For the lower bound, isotropy and subadditivity give~$\E[T^{(i)}(u,v)]\geq s+(\E[T_s]-s)/(1+\sup\rho)$ for all points~$u$ and~$v$ on the faces of the~$i$th slab. By Steps~1--3, $\E[\sup_{u,v}|T^{(i)}(u,v)-\E[T^{(i)}(u,v)]|]\leq\kappa Q_s$, with~$\kappa$ small, where~$u$ and~$v$ range over two cubes of side~$(sQ_s)^{\nf12}$ in the faces of the~$i$th slab. By Step~1, a minimizing chain moves by order~$(sQ_s)^{\nf12}$ per slab on average, so along it the fluctuations add up to at most~$C\kappa\lceil s\rceil Q_s$ in mean. The upper bound follows from~\eqref{eq:F} for this row and the bound~$\E[T_t]-t\leq C(t\log(2+t))^{\nf12}$ of Howard and Newman~\cite[Lemmas~4.1 and~4.2]{HN} at~$t=s\lceil s\rceil$. Since~$\log^2(2+s)$ is much smaller than~$Q_s$, the two bounds force~$\E[T_s]-s\leq Q_s/100$ when~$\kappa$ is small.

In isotropic random-graph models with exponential tails at all distances, Alexander~\cite[Theorem~1.5]{Alexander20} bounds the nonrandom fluctuations by a large constant times the fluctuation scale. The induction needs the much smaller constant~$\nf1{100}$, but only at one distance, and the small variance of Step~3 provides it.

\subsection{Notation and conventions}\label{ssec:notation}

We use the notation of~\cite[(0.41)--(0.42)]{AKM} for exponential bounds and its properties from~\cite[Appendix~A]{AKM}. For~$s,\theta>0$ and a random variable~$X$, we write
\begin{equation*}
\begin{aligned}
X\leq\Os_s(\theta)&\quad\iff\quad\E\bigl[\exp\bigl((X_+/\theta)^s\bigr)\bigr]\leq2\,,\qquad X_+\coloneqq\max\{X,0\}\,,\\
X\leq Y+\Os_s(\theta)\ \text{ and }\ Y\geq X-\Os_s(\theta)&\quad\iff\quad X-Y\leq\Os_s(\theta)\,.
\end{aligned}
\end{equation*} If~$X\leq\Os_s(\theta)$, then~$\Pp[X\geq\theta x]\leq2e^{-x^s}$ for~$x\geq0$ and~$\E[X_+^p]\leq C(p,s)\theta^p$ for~$p\geq1$. If~$X_j\leq\Os_s(\theta)$ for~$j\leq J$, then
\begin{equation}\label{eq:union}
\max_{j\leq J}X_j\leq\Os_s\bigl(C(s)\theta\log^{\nf1s}(2+J)\bigr)\,.
\end{equation}
If~$s\geq1$ and~$0\leq X_j\leq\Os_s(\theta_j)$ for every~$j$, then
\begin{equation}\label{eq:sum}
\sum_jX_j\leq\Os_s\Bigl(\sum_j\theta_j\Bigr)\,.
\end{equation}

All constants and thresholds may depend on~$d$ and the law of the density, and we indicate only further dependence, as in~$C(L)$. The constants~$c,C>0$ may change from line to line. We write~$[u,v]$ for the segment from~$u$ to~$v$, and~$\angle(x,y)\in[0,\pi]$ for the angle between nonzero vectors~$x$ and~$y$. We write~$e_1,\ldots,e_d$ for the standard basis of~$\R^d$ and~$\mathbf 1_E$ for the indicator of an event~$E$.

\subsection{Formalization}\label{ssec:formalization}

All the results of this paper are formalized in the Lean~4 proof assistant, with the Mathlib library, in the repository
\begin{center}
\repourl
\end{center}
It also contains the companion paper~\cite{ABlattice} on lattice first-passage percolation, with the formalization of its main theorem, and checked lists of the errors and gaps found in~\cite{BSS23} and~\cite{BSS}.

\subsection{Use of artificial intelligence}\label{ssec:ai}

This paper grew out of a project, begun on October~21, 2026, to formalize the two papers of Basu, Sidoravicius and Sly~\cite{BSS23,BSS} in the Lean~4 proof assistant with the help of AI agents. Checking their arguments in the detail that a formalization requires, the agents found a number of points that needed correction. Many have local corrections, but we did not obtain a corrected proof. While searching for another route, we found a much simpler argument, which is the one presented here and applies in every dimension. Its induction on scales comes from the work of Basu, Sidoravicius and Sly, and the present paper would not exist without their pioneering work.

An eight-page proof for a planar model was obtained by GPT-6 Astra agents, and its extension to every dimension by a GPT-6.1 Sol agent. An initial draft of the present paper was written by Claude Opus~5.5 under our direction, and the authors subsequently heavily revised and rewrote this draft.
The formalization described in Section~\ref{ssec:formalization} was written with the help of AI agents.

\subsection*{Acknowledgments}

S.~A. acknowledges support from the European Research Council (ERC) under the European Union's Horizon Europe research and innovation programme, grant agreement number 101200828.

\section{Preliminaries}\label{sec:prelim}

This section collects classical bounds on passage times and near-geodesics, the comparison of slab passage times with~$T$, and a bound for sums of independent arrays along a sequence. Throughout,~$\rho$ satisfies Assumption~\ref{ass:local} and, as in Section~\ref{ssec:overview}, $\mu=1$.

\begin{lemma}\label{lem:classical}
For every~$t\geq1$,
\begin{equation}\label{eq:classical}
|T_t-\E[T_t]|\leq\Os_2(Ct^{\nf12})\qquad\text{and}\qquad
0\leq\E[T_t]-t\leq C(t\log(2+t))^{\nf12}\,.
\end{equation}
\end{lemma}

\begin{proof}
Both tails of the first bound follow from the inequalities of Boucheron, Lugosi and Massart~\cite[Corollary~3 and Theorem~7]{BLM03}, since resampling~$\omega$ in one unit cube changes~$T_t$ by at most~$C$, and the squares of the increases of~$T_t$ sum over the cubes to at most~$Ct$. The second bound follows from the first as in Howard and Newman~\cite[Lemmas~4.1 and~4.2]{HN}.
\end{proof}

Lemma~\ref{lem:classical} implies that, away from its endpoints, a near-geodesic lies in narrow cones about the segment between them, uniformly over the endpoints in a ball.

\begin{lemma}\label{lem:straight}
Let~$r\geq1$. There exists a random variable~$W$ with~$\log(2+r)\leq W\leq\Os_2(C\log(2+r))$ such that, for all points~$a$ and~$b$ in the ball of radius~$r$ around~$0$, every point~$x$ of a near-geodesic from~$a$ to~$b$ with~$|x-a|\geq CW^2$ satisfies~$\angle(x-a,b-a)\leq\nf1{10}$.
\end{lemma}

\begin{proof}
We use Lemma~\ref{lem:classical} to make the straightness argument of Howard and Newman~\cite[Theorem~2.6 and Lemma~2.7]{HN} uniform in the endpoints. The near-geodesics between points of the ball lie in the ball of radius~$Cr$ around~$0$. Let~$W\geq\log(2+r)$ be the smallest number such that~$|T(x,y)-|y-x||\leq C(W|y-x|^{\nf12}+W^2)$ for all~$x$ and~$y$ in this larger ball. By Lemma~\ref{lem:classical}, isotropy, the union bound~\eqref{eq:union} over the pairs of integer points in it, and~$|T(x,y)-T(x',y')|\leq C(|x-x'|+|y-y'|)$, we have~$W\leq\Os_2(C\log(2+r))$. With this bound in place of their fluctuation bound, their argument proves the lemma.
\end{proof}

The proof uses the minimum~\eqref{eq:F} for several numbers of slabs. For a positive integer~$M$ and~$s>2$, let~$F$ be the minimum~\eqref{eq:F} over the chains across the~$M$ slabs~$[(i-1)s,is]\times\R^{d-1}$, with~$\rho^{(i)}$ and~$T^{(i)}$ the~$\rho_P$ and~$T_P$ of the~$i$th slab. For~$\tau\in\R^d$, let~$F_\tau$ be defined by~\eqref{eq:F} with the slabs, the endpoints and the points~$x_i$ translated by~$\tau$. The next lemma compares slab passage times with~$T$, and proves the comparison of~$F$ with~$T_{Ms}$ stated in Section~\ref{ssec:overview}.

\begin{lemma}\label{lem:slab}
Let~$P$ be a slab of width~$s>2$, and let~$z\in\R^d$. Then
\begin{equation}\label{eq:slab}
0\leq\sup_{x,y}\bigl(T_P(x,y)-T(x,y)\bigr)\leq\Os_1(C\log^2(2+s))\,,
\end{equation}
where~$x$ and~$y$ range over the points on different faces of~$P$ with~$|x-z|\leq s$ and~$|y-x|\leq5^{\nf12}\,s$. Moreover,
\begin{equation}\label{eq:columns}
0\leq F-T_{Ms}\leq\Os_1(CM\log^2(2+Ms))\,,
\end{equation}
and the same bound holds for~$F_\tau-T(\tau,\tau+Mse_1)$, for every~$\tau\in\R^d$.
\end{lemma}

\begin{proof}
By stationarity, we may assume that~$z=0$ and~$\tau=0$. The lower bounds follow from~$T_P\geq T$, which holds since~$\rho_P\geq\rho$ on~$P$, and from the triangle inequality. Let~$W$ be as in Lemma~\ref{lem:straight}, with~$r\coloneqq4s$ for~\eqref{eq:slab} and~$r\coloneqq CMs$ for~\eqref{eq:columns}, so that~$W^2\leq\Os_1(C\log^2(2+r))$.

For~\eqref{eq:slab}, suppose first that~$s\geq CW^2$. Since~$|y-x|\leq5^{\nf12}\,s$, the segment~$[x,y]$ makes an angle at most~$\arctan2$ with the normal of~$P$. Hence, by Lemma~\ref{lem:straight} from~$x$ and from~$y$, the points of a near-geodesic from~$x$ to~$y$ at distance at least~$CW^2$ from~$x$ and~$y$ lie at distance more than one from~$\partial P$. Replacing its two end pieces by segments gives~$T_P(x,y)\leq T(x,y)+CW^2$. If~$s<CW^2$, the segment~$[x,y]$ gives the same bound. This proves~\eqref{eq:slab}.

For~\eqref{eq:columns}, suppose first that~$s\geq CW^2$. Let~$x_i$ be the first point of a near-geodesic from~$0$ to~$Mse_1$ on the face~$\{x\cdot e_1=is\}$, for~$0<i<M$, and let~$x_0\coloneqq0$ and~$x_M\coloneqq Mse_1$. Since subarcs of near-geodesics are near-geodesics, Lemma~\ref{lem:straight}, applied to the subarc from~$x_{i-1}$ to~$Mse_1$ and to the reversed curve, gives~$\angle(x_i-x_{i-1},e_1)\leq\nf15$. Since~$\tan\nf15<\nf14$, the~$x_i$ are admissible in~\eqref{eq:F}. The bound~$T_P\leq T+CW^2$ in each slab then gives~$F\leq T_{Ms}+1+CMW^2$. If~$s<CW^2$, the points~$ise_1$ give~$F\leq CMW^2$. This proves~\eqref{eq:columns}.
\end{proof}

Moment bounds for fixed indices do not directly bound a sum along a random sequence of indices. The next lemma allows random sequences, at the cost of their squared increments.

\begin{lemma}\label{lem:polynomialarray}
Let~$M\geq1$ be an integer, and let~$Z_1,\ldots,Z_M$ be independent random arrays of nonnegative random variables~$Z_i(j,k)$, indexed by~$j,k\in\Z^{d-1}$, with~$\E[Z_i(j,k)^{4d}]\leq1$ for all~$i,j,k$. Then, with~$k_0\coloneqq0$ and a constant~$C$ depending only on~$d$,
\begin{equation*}
\E\Bigl[\sup_{k_1,\ldots,k_M\in\Z^{d-1}}\Bigl\{\sum_{i=1}^MZ_i(k_{i-1},k_i)-\sum_{i=1}^M|k_i-k_{i-1}|^2\Bigr\}\Bigr]\leq CM\,.
\end{equation*}
\end{lemma}

\begin{proof}
For~$k=(k_1,\ldots,k_M)$ in~$(\Z^{d-1})^M$, with~$k_0=0$, let
\begin{equation*}
S(k)\coloneqq\sum_{i=1}^MZ_i(k_{i-1},k_i)\qquad\text{and}\qquad P(k)\coloneqq\sum_{i=1}^M|k_i-k_{i-1}|^2\,,
\end{equation*}
so that the lemma bounds~$\E[\sup_k(S(k)-P(k))]$. Let~$V_r\coloneqq\max\{S(k):P(k)\leq2^rM\}$ for each integer~$r\geq0$. If~$2^{r-1}M<P(k)\leq2^rM$, then~$S(k)-P(k)\leq V_r-2^{r-1}M$. Hence
\begin{equation*}
\E\Bigl[\sup_k\bigl(S(k)-P(k)\bigr)\Bigr]\leq\E[V_0]+\sum_{r\geq1}\E\bigl[(V_r-2^{r-1}M)_+\bigr]\leq\E[V_0]+\sum_{r\geq1}\frac{\E[V_r^2]}{2^{r-1}M}\,,
\end{equation*}
and it suffices to show that~$\E[V_r^2]^{\nf12}\leq C2^{\nf r4}M$ for every~$r$.

Fix~$r$. For each~$k$, we write~$Z_{(1)}(k)\geq\cdots\geq Z_{(M)}(k)$ for the entries~$Z_i(k_{i-1},k_i)$ in decreasing order. For~$1\leq a\leq M$, let~$Y_a\coloneqq\max\{Z_{(a)}(k):P(k)\leq2^rM\}$, so that~$V_r\leq\sum_aY_a$. For~$\lambda>0$, if~$Y_a\geq\lambda$, then~$Z_{(a)}(k)\geq\lambda$ for some~$k$ with~$P(k)\leq2^rM$, so~$a$ of the entries~$Z_i(k_{i-1},k_i)$ are at least~$\lambda$. These entries occupy~$a$ triples~$(i,k_{i-1},k_i)$, which are determined by their indices~$i$ and by~$2a$ displacements in~$\Z^{d-1}$ of total length at most~$\sum_i|k_i-k_{i-1}|\leq(MP(k))^{\nf12}\leq2^{\nf r2}M$. Hence the set of these~$a$ triples takes at most~$(eM/a)^a(C2^{\nf r2}M/a)^{2a(d-1)}$ values. For each such set, the~$a$ entries at its triples lie in distinct arrays, so, by independence and Markov's inequality for their~$4d$th moments, they are all at least~$\lambda$ with probability at most~$\lambda^{-4da}$. A union bound over these sets therefore gives
\begin{equation*}
\Pp[Y_a\geq\lambda]\leq\Bigl(\frac{eM}a\Bigr)^a\Bigl(\frac{C2^{\nf r2}M}a\Bigr)^{2a(d-1)}\lambda^{-4da}\leq\Bigl(\frac{C2^{\nf r4}(M/a)^{\nf12}}\lambda\Bigr)^{4da}\,,
\end{equation*}
so~$\E[Y_a^2]^{\nf12}\leq C2^{\nf r4}(M/a)^{\nf12}$, and summing over~$a$ by Minkowski's inequality gives~$\E[V_r^2]^{\nf12}\leq C2^{\nf r4}M$.
\end{proof}
\section{The induction}\label{sec:induction}

We prove Theorem~\ref{thm:general} by induction on the length. For an integer~$K\geq10$ and~$A\geq1$, let~$Q_s\coloneqq As^{\nf12-\nf1K}$ for~$s>0$, as in Section~\ref{ssec:overview}. For a length~$N\geq2$, the \emph{induction hypothesis} is
\begin{equation}\label{eq:premise}
|T_s-s|\leq\Os_1(Q_s)\qquad\text{for }1\leq s\leq N/2\,,
\end{equation}
which is~\eqref{eq:hypothesis} for these~$s$, and the conclusion of the induction step is
\begin{equation}\label{eq:improvement}
|T_N-N|\leq\Os_1(Q_N/2)\,.
\end{equation}

\begin{proposition}[Induction step]\label{prop:induction}
There exist an integer~$K\geq10$ and~$N_0\geq2$ such that~\eqref{eq:premise} implies~\eqref{eq:improvement} for every~$A\geq1$ and every~$N\geq N_0$.
\end{proposition}

\begin{proof}[Proof of Theorem~\ref{thm:general}, given Proposition~\ref{prop:induction}]
This is immediate by Proposition~\ref{prop:induction} and induction on the length, with~$\varepsilon=\nf2K$ and~$\mu A$ as the constant of Theorem~\ref{thm:general}, since~$|T_s-s|\leq Cs$ gives~$|T_s-s|\leq\Os_1(Q_s)$ for~$s\in[1,N_0]$ once~$A$ is large, depending on~$K$.
\end{proof}

The rest of the paper proves Proposition~\ref{prop:induction}. By~\eqref{eq:classical}, $|T_N-N|\leq\Os_1(C(N\log(2+N))^{\nf12})$, which gives~\eqref{eq:improvement} if~$Q_N\geq2C(N\log(2+N))^{\nf12}$. We may therefore assume that, with this constant~$C$,
\begin{equation}\label{eq:regime}
Q_N\leq2C(N\log(2+N))^{\nf12}\,.
\end{equation}
From now on, we assume~\eqref{eq:premise} and~\eqref{eq:regime}.
\section{Passage times between cubes}\label{sec:exponential}

In this section we extend the induction hypothesis~\eqref{eq:premise} in two ways, for arbitrary numbers~$0<\beta_1\leq\beta_2$: to bounds uniform over the pairs of points in two cubes on opposite faces of a slab of width~$s\in[\beta_1N,N]$, and to lengths~$t\in[\beta_1N,\beta_2N]$. For~$s>2$, a rigid motion~$g$, numbers~$\sigma,\sigma'\in[0,4(sQ_s)^{\nf12}]$ and~$y\in\R^{d-1}$, let
\begin{equation}\label{eq:cubepair}
\begin{gathered}
I\coloneqq g\bigl(\{0\}\times[-\sigma/2,\sigma/2]^{d-1}\bigr)\,,\qquad I'\coloneqq g\bigl(\{s\}\times(y+[-\sigma'/2,\sigma'/2]^{d-1})\bigr)\,,\\
S\coloneqq g\bigl([0,s]\times\R^{d-1}\bigr)\,.
\end{gathered}
\end{equation}

We first extend~\eqref{eq:premise} from points to boxes by chaining over dyadic subboxes. A \emph{box} is the image under a rigid motion of a product of~$d$ compact intervals, some of which may be single points.

\begin{lemma}\label{lem:tree}
For~$16\leq t\leq N/2$ and every box~$B$ centered at~$te_1$, of diameter at most~$t/8$, contained in the~$16(d\,tQ_t)^{\nf12}$-neighborhood of the line~$\R e_1$,
\begin{equation}\label{eq:tree}
\sup_{z\in B}\{T(0,z)-z\cdot e_1\}\leq\Os_1(CQ_t)\,.
\end{equation}
By stationarity and isotropy, the same holds for~$T(x,z)-(z-x)\cdot e$ over boxes centered at~$x+te$ contained in the~$16(d\,tQ_t)^{\nf12}$-neighborhood of~$x+\R e$, for every~$x\in\R^d$ and unit vector~$e$.
\end{lemma}

\begin{proof}
For an integer~$j\geq0$, the \emph{dyadic subboxes} of~$B$ of generation~$j$ are the~$2^{dj}$ boxes obtained by cutting each edge of~$B$ into~$2^j$ intervals of equal length. For such a subbox~$D$ with center~$c$, let~$p(D)\coloneqq c-2^{-j}te_1$, so that~$p(B)=0$. Let~$\mathcal A_j$ be the set of the~$2^{d(j+1)}$ pairs~$(p(D),p(D'))$, where~$D$ has generation~$j$ and~$D'\subseteq D$ has generation~$j+1$. For~$(y,w)\in\mathcal A_j$, the vector~$w-y$ is~$2^{-j-1}$ times the vector from~$0$ to a vertex of~$B$, so
\begin{equation}\label{eq:treestep}
\frac{15}{32}2^{-j}t\leq|w-y|\leq2^{-j}t\qquad\text{and}\qquad|w-y|-(w-y)\cdot e_1\leq C2^{-j}Q_t\,.
\end{equation}
Let~$j_0$ be the largest integer with~$2^{-j_0}t\geq4$. For~$j<j_0$, the pairs in~$\mathcal A_j$ satisfy~$1\leq|w-y|\leq N/2$. Hence~\eqref{eq:premise}, isotropy, $Q_{2^{-j}t}\leq2^{-\nf{2j}5}Q_t$ and the union bound~\eqref{eq:union} show that the maximum~$X_j$ of~$(T(y,w)-|w-y|)_+$ over~$(y,w)\in\mathcal A_j$ satisfies~$X_j\leq\Os_1(C(1+j)2^{-\nf{2j}5}Q_t)$.

Let~$z\in B$, let~$D_0=B\supseteq D_1\supseteq\cdots$ be nested subboxes containing~$z$, with~$D_j$ of generation~$j$, and let~$y_j\coloneqq p(D_j)$. Since~$y_0=0$ and~$|z-y_{j_0}|\leq9$, the triangle inequality and~\eqref{eq:treestep} give
\begin{equation*}
T(0,z)-z\cdot e_1\leq\sum_{j<j_0}\bigl(T(y_j,y_{j+1})-(y_{j+1}-y_j)\cdot e_1\bigr)+C\leq\sum_{j<j_0}\bigl(X_j+C2^{-j}Q_t\bigr)+C\,.
\end{equation*}
The right side does not depend on~$z$, and, by~\eqref{eq:sum}, it is at most~$\Os_1(CQ_t)$.
\end{proof}

We prove the next proposition with Lemma~\ref{lem:tree} and the triangle inequality through points of~$\R e_1$, first for~$s\leq N/4$, where~\eqref{eq:premise} applies at all the lengths involved, and then for~$s>N/4$, after bounding~$T_t$ for every~$t\in[\beta_1N,\beta_2N]$.

\begin{proposition}[Passage times between cubes]\label{prop:profiles}
If~$N$ is sufficiently large, depending only on~$K$ and~$\beta_1$, then, for every~$s\in[\beta_1N,N]$ and all cubes~$I,I'$ as in~\eqref{eq:cubepair} with~$|y|\leq s$,
\begin{equation}\label{eq:profiles}
\sup_{u\in I,\,v\in I'}|T_P(u,v)-|v-u||\leq\Os_1(CQ_s)\,,
\end{equation}
both for~$P=\R^d$ and for~$P=S$.
\end{proposition}

\begin{proof}[Proof of Proposition~\ref{prop:profiles} for~$s\leq N/4$]
By stationarity and isotropy, we may assume that the centers of~$I$ and~$I'$ are~$0$ and~$re_1$, with~$s\leq r\leq2^{\nf12}s$. By the triangle inequality, for~$u\in I$ and~$v\in I'$,
\begin{equation*}
\begin{aligned}
T(u,v)&\geq T\bigl(-\nf s{16}e_1,(r+\nf s{16})e_1\bigr)-T\bigl(-\nf s{16}e_1,u\bigr)-T\bigl(v,(r+\nf s{16})e_1\bigr)\,,\\
T(u,v)&\leq T\bigl(u,\nf s{16}e_1\bigr)+T\bigl(\nf s{16}e_1,(r-\nf s{16})e_1\bigr)+T\bigl((r-\nf s{16})e_1,v\bigr)\,.
\end{aligned}
\end{equation*}
Since~$r+s/8<2s\leq N/2$, the induction hypothesis~\eqref{eq:premise} gives~$|T(x,z)-|z-x||\leq\Os_1(CQ_s)$ for the two terms~$T(x,z)$ of the display with~$x,z\in\R e_1$. Since~$Q_{s/16}\geq Q_s/4$, Lemma~\ref{lem:tree} with~$t=s/16$ and~$e=\pm e_1$ gives~$T(x,z)\leq(z-x)\cdot e_1+\Os_1(CQ_s)$ for the four terms of the display that involve~$u$ or~$v$, uniformly in~$u$ and~$v$, for~$N$ large. In each line of the display, the terms~$(z-x)\cdot e_1$ and~$|z-x|$ of these bounds add up to~$(v-u)\cdot e_1$, so, by~\eqref{eq:sum}, $|T(u,v)-(v-u)\cdot e_1|\leq\Os_1(CQ_s)$ uniformly in~$u$ and~$v$. Since~$0\leq|v-u|-(v-u)\cdot e_1\leq CQ_s$, this proves~\eqref{eq:profiles} for~$P=\R^d$. For~$N$ large, $|v-u|\leq5^{\nf12}\,s$, so~\eqref{eq:slab} with~$P=S$ and~$z=0$ adds~$\Os_1(C\log^2(2+N))\leq\Os_1(Q_s)$.
\end{proof}

We now cut the faces of the slabs of~\eqref{eq:F} into cubes of side~$(sQ_s)^{\nf12}$ (Figure~\ref{fig:stepone}). For~$M$, $s$ and~$F$ as in Section~\ref{sec:prelim}, let
\begin{equation}\label{eq:cubes}
I_i(j)\coloneqq\{is\}\times(sQ_s)^{\nf12}\bigl(j+[0,1]^{d-1}\bigr)\qquad\text{for }0\leq i\leq M\text{ and }j\in\Z^{d-1}\,,
\end{equation}
and let~$x_0,\ldots,x_M$ be the lexicographically least chain attaining the minimum in~\eqref{eq:F}. Translating~$\omega$ and the endpoints by a vector orthogonal to~$e_1$ translates the~$x_i$ by the same vector. For an array~$Z_i(j,k)$ indexed by~$1\leq i\leq M$ and~$j,k\in\Z^{d-1}$, we write~$Z_i(u,v)\coloneqq Z_i(j,k)$ for the lexicographically greatest~$j$ and~$k$ with~$u\in I_{i-1}(j)$ and~$v\in I_i(k)$.

The next lemma bounds~$F$ and the transverse increments of its minimizing chain, by a union bound over the sequences~$(k_i)_i$ for which a chain visits the cubes~$I_i(k_i)$.

\begin{lemma}\label{lem:rowenergy}
If~$N$ is sufficiently large, depending only on~$K$ and~$\beta_1$, then, for every positive integer~$M$ and every~$s\in[\beta_1N,N/4]$,
\begin{equation}\label{eq:rowenergy}
|F-Ms|+\frac1s\sum_{i=1}^M|x_i-x_{i-1}-se_1|^2\leq CMQ_s+\Os_1(CQ_s)\,.
\end{equation}
\end{lemma}

\begin{proof}
For~$1\leq i\leq M$ and~$j,k\in\Z^{d-1}$ with~$(sQ_s)^{\nf12}|k-j|\leq s/2$, let~$u\in I_{i-1}(j)$ and~$v\in I_i(k)$, so that these cubes are as in~\eqref{eq:cubepair} with~$\sigma=\sigma'=(sQ_s)^{\nf12}$, with~$S$ the~$i$th slab and with~$|y|\leq s/2$. For~$N$ large, $|v-u|\geq s+cQ_s|k-j|^2-CQ_s$. Hence Proposition~\ref{prop:profiles}, with~$P=S$ the~$i$th slab, gives
\begin{equation*}
Z_i(j,k)\coloneqq\Bigl(s+cQ_s|k-j|^2-\inf_{I_{i-1}(j)\times I_i(k)}T^{(i)}\Bigr)_+\leq\Os_1(C_1Q_s)\,.
\end{equation*}
The arrays~$Z_i$ are independent over~$i$, so, for~$k_0=0,k_1,\ldots,k_M$ with~$(sQ_s)^{\nf12}|k_i-k_{i-1}|\leq s/2$,
\begin{equation*}
\E\Bigl[\exp\Bigl(\sum_{i=1}^M\frac{Z_i(k_{i-1},k_i)-\frac c2Q_s|k_i-k_{i-1}|^2}{C_1Q_s}\Bigr)\Bigr]\leq\prod_{i=1}^M2e^{-c|k_i-k_{i-1}|^2/(2C_1)}\,.
\end{equation*}
The sum of these products over~$k_1,\ldots,k_M\in\Z^{d-1}$ is at most~$C^M$. Hence a union bound gives
\begin{equation*}
\sup_{k_1,\ldots,k_M}\Bigl\{Ms+\frac c2Q_s\sum_{i=1}^M|k_i-k_{i-1}|^2-\sum_{i=1}^M\inf_{I_{i-1}(k_{i-1})\times I_i(k_i)}T^{(i)}\Bigr\}\leq CMQ_s+\Os_1(CQ_s)\,.
\end{equation*}
The same argument, applied to the single chain~$x_i=ise_1$, gives~$F\leq Ms+CMQ_s+\Os_1(CQ_s)$. Let~$k_i$ be the lexicographically greatest index with~$x_i\in I_i(k_i)$; since~$|x_i-x_{i-1}-se_1|\leq s/4$, we have~$(sQ_s)^{\nf12}|k_i-k_{i-1}|\leq s/2$ for~$N$ large. Then
\begin{equation*}
F\geq\sum_{i=1}^M\inf_{I_{i-1}(k_{i-1})\times I_i(k_i)}T^{(i)}\qquad\text{and}\qquad|x_i-x_{i-1}-se_1|^2\leq2sQ_s\bigl(|k_i-k_{i-1}|^2+d-1\bigr)\,,
\end{equation*}
so the bound on the supremum gives~\eqref{eq:rowenergy}.
\end{proof}

For~$s$ and~$N$ as in Lemma~\ref{lem:rowenergy}, let~$Z_1,\ldots,Z_M$ be independent random arrays of nonnegative random variables with~$\E[(Z_i(j,k)-|j-k|^2)_+^{4d}]\leq1$ for all~$i,j,k$. With~$k_i$ the lexicographically greatest index with~$x_i\in I_i(k_i)$, \eqref{eq:rowenergy} gives~$\E[\sum_i|k_i-k_{i-1}|^2]\leq CM$, so Lemma~\ref{lem:polynomialarray}, applied to the arrays~$(Z_i(j,k)-|j-k|^2)_+$, gives
\begin{equation}\label{eq:chainsum}
\E\Bigl[\sum_{i=1}^MZ_i(x_{i-1},x_i)\Bigr]\leq CM\,.
\end{equation}

The next lemma gives the bound~(a) of Section~\ref{ssec:overview}, from Lemma~\ref{lem:rowenergy} and~\eqref{eq:columns} with a fixed number of slabs.

\begin{lemma}\label{lem:parent}
If~$N$ is sufficiently large, depending only on~$K,\beta_1,\beta_2$, then, for every~$t\in[\beta_1N,\beta_2N]$,
\begin{equation}\label{eq:parent}
|T_t-t|\leq C(\beta_2)Q_t+\Os_1(Q_t/100)\,.
\end{equation}
\end{lemma}

\begin{proof}
Fix an integer~$M$, depending only on~$\beta_2$, such that~$s\coloneqq t/M\leq N/4$ and~$C_1M^{-\nf25}\leq\nf1{200}$, where~$C_1$ is the constant of the term~$\Os_1(CQ_s)$ in~\eqref{eq:rowenergy}. Since~$s\geq\beta_1N/M$, Lemma~\ref{lem:rowenergy} with~$\beta_1/M$ in place of~$\beta_1$, and~\eqref{eq:columns}, give~$|T_t-t|\leq CMQ_s+\Os_1(C_1Q_s+CM\log^2(2+N))$. Since~$Q_s\leq M^{-\nf25}Q_t$ and~$CM\log^2(2+N)\leq Q_t/200$ for~$N$ large, this proves~\eqref{eq:parent}.
\end{proof}

\begin{proof}[Proof of Proposition~\ref{prop:profiles} for~$N/4<s\leq N$]
The proof for~$s\leq N/4$ applies, with~\eqref{eq:parent} in place of~\eqref{eq:premise} at the lengths~$r\pm s/8\in[N/8,2N]$.
\end{proof}
\section{Transverse fluctuations and continuity in the endpoints}\label{sec:localization}
\suppressfloats[t]

We fix a length~$n\in[\beta_1N,N]$, and cubes~$I,I'$ and a slab~$S$ as in~\eqref{eq:cubepair} with~$s=n$ and~$|y|\leq n/2$. We bound the distance from the curves in~$\Pi_{T_S}(u,v)$, for~$u\in I$ and~$v\in I'$, to~$[u,v]$, and the modulus of continuity of~$T_S(u,v)-|v-u|$ in~$u$ and~$v$.

To bound the distance, we extend the curves to two fixed points~$a$ and~$b$ on the line through the centers of the cubes. Since the limit shape is a ball, the results of Section~\ref{sec:exponential} show that an extended curve through a point at distance~$\delta$ from this line has passage time at least~$|b-a|+c\delta^2/n-\Os_1(CQ_n)$, uniformly in the curve.

\begin{lemma}[Transverse fluctuations]\label{lem:localization}
If~$N$ is sufficiently large, depending only on~$K$ and~$\beta_1$, then
\begin{equation}\label{eq:localization}
(nQ_n)^{-\nf12}\sup_{u\in I,v\in I'}\sup_{\pi\in\Pi_{T_S}(u,v)}\sup_{x\in\pi}\dist(x,[u,v])\leq\Os_2(C)\,.
\end{equation}
\end{lemma}

Since~$\rho$ is bounded, moving the endpoints by~$h(nQ_n)^{\nf12}$ changes~$T_S(u,v)$ by at most~$Ch(nQ_n)^{\nf12}$, which is much larger than~$Q_n$. The next lemma shows that~$T_S(u,v)-|v-u|$ changes only by a small multiple of~$Q_n$; its proof reroutes a near-geodesic near its endpoints, where fluctuation bounds at shorter distances apply.

\begin{lemma}[Continuity in the endpoints]\label{lem:modulus}
For every~$\eta>0$, there exists~$h>0$, depending only on~$\eta$, such that, if~$N$ is sufficiently large, depending only on~$K,\eta,\beta_1$, then
\begin{equation}\label{eq:modulus}
\E\Bigl[\sup_{\substack{u,\widetilde u\in I,\ v,\widetilde v\in I'\\|u-\widetilde u|,\,|v-\widetilde v|\leq h(nQ_n)^{\nf12}}}\bigl|T_S(u,v)-T_S(\widetilde u,\widetilde v)-|v-u|+|\widetilde v-\widetilde u|\bigr|^{8d}\Bigr]^{\nf1{8d}}\leq\eta Q_n\,.
\end{equation}
\end{lemma}

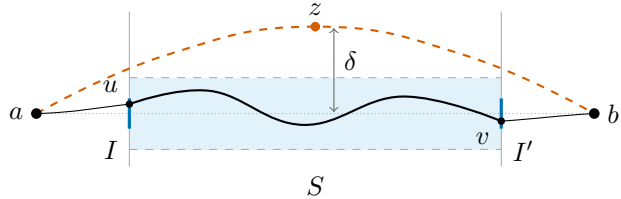
\begin{figure}[b]
\centering
\begin{tikzpicture}[x=.82cm,y=.5cm,every node/.style={font=\small}]
\fill[figshade] (0,-.95) rectangle (6,.95);
\draw[dashed,figguide] (0,-.95)--(6,-.95) (0,.95)--(6,.95);
\draw[figguide] (0,-1.4)--(0,2.7) (6,-1.4)--(6,2.7);
\draw[figguide,densely dotted] (-1.5,0)--(7.5,0);
\draw (-1.5,0) .. controls (-.9,.02) and (-.4,.18) .. (0,.25);
\draw[thick] plot[smooth,tension=.7] coordinates {(0,.25) (1.3,.6) (2.8,-.3) (4.4,.45) (6,-.2)};
\draw (6,-.2) .. controls (6.5,-.15) and (7,-.04) .. (7.5,0);
\draw[very thick,figcube] (0,-.4)--(0,.4) (6,-.4)--(6,.4);
\draw[fighigh,dashed,thick] plot[smooth,tension=.8] coordinates {(-1.5,0) (.6,1.55) (3,2.3) (5.4,1.55) (7.5,0)};
\fill (-1.5,0) circle (2pt) (7.5,0) circle (2pt) (0,.25) circle (1.4pt) (6,-.2) circle (1.4pt);
\fill[fighigh] (3,2.3) circle (1.8pt);
\draw[<->,black!60] (3.3,.04)--(3.3,2.26);
\node[right] at (3.3,1.4) {$\delta$};
\node[above] at (3,2.35) {$z$};
\node[left] at (-1.55,0) {$a$};
\node[right] at (7.55,0) {$b$};
\node[above left] at (-.02,.3) {$u$};
\node[below left] at (5.98,-.22) {$v$};
\node[below left] at (-.02,-.45) {$I$};
\node[below right] at (6.02,-.45) {$I'$};
\node[below] at (3,-1.4) {$S$};
\end{tikzpicture}
\caption{The slab~$S$ with the cubes~$I$ and~$I'$ (thick blue) on its faces, a curve in~$\Pi_{T_S}(u,v)$ (black) extended to~$a$ and~$b$ (thin black), and a curve from~$a$ to~$b$ (orange dashed) through a point~$z$ at distance~$\delta$ from the line through~$a$ and~$b$ (dotted), outside the shaded band of width~$C(nQ_n)^{\nf12}$.}\label{fig:localization}
\end{figure}

\begin{proof}[Proof of Lemma~\ref{lem:localization}]
By stationarity and isotropy, we may assume that the centers of~$I$ and~$I'$ are~$0$ and~$re_1$, with~$n\leq r\leq2n$. Write~$z=(z_1,w)$ with~$z_1\coloneqq e_1\cdot z$ and~$w\in\R^{d-1}$. The points of~$I$ and~$I'$ lie within distance~$2(d\,nQ_n)^{\nf12}$ of~$0$ and~$re_1$, and the curves in~$\Pi_{T_S}(u,v)$ have length at most~$C(|v-u|+1)$. Hence, for~$N$ large, the curves in~\eqref{eq:localization} lie in
\begin{equation*}
\Lambda\coloneqq\{(z_1,w):-C_0n\leq z_1\leq r+C_0n,\ |w|\leq C_0n\}
\end{equation*}
for a constant~$C_0$. Let~$a\coloneqq-2C_0ne_1$ and~$b\coloneqq(r+2C_0n)e_1$. For a rectifiable curve~$\pi$, let~$T(\pi)\coloneqq\int_\pi\rho(z)\,|dz|$. We compare an upper bound on~$T(a,u)+T(\pi)+T(v,b)$ (the black curves in Figure~\ref{fig:localization}) with a lower bound on~$T(a,z)+T(z,b)$ for~$z\in\Lambda$ (the orange dashed curve).

We first prove the upper bound
\begin{equation*}
\sup_{u\in I,\,v\in I'}\ \sup_{\pi\in\Pi_{T_S}(u,v)}\bigl\{T(a,u)+T(\pi)+T(v,b)\bigr\}\leq|b-a|+\Os_1(CQ_n)\,.
\end{equation*}
Since~$\rho\leq\rho_S$, we have~$T(\pi)\leq T_S(u,v)+1\leq|v-u|+\Os_1(CQ_n)$, by~\eqref{eq:profiles}. By the triangle inequality through the points~$-\nf n{16}e_1$ and~$(r+\nf n{16})e_1$, Lemma~\ref{lem:tree} at the length~$n/16$ and~\eqref{eq:parent} at the length~$(2C_0-\nf1{16})n$ give~$T(a,u)\leq e_1\cdot(u-a)+\Os_1(CQ_n)$ and~$T(v,b)\leq e_1\cdot(b-v)+\Os_1(CQ_n)$, uniformly in~$u$ and~$v$. Since~$|v-u|-e_1\cdot(v-u)\leq CQ_n$, the upper bound follows by~\eqref{eq:sum}.

We next prove the lower bound: for some~$c>0$,
\begin{equation}\label{eq:boxpenalty}
\sup_{z=(z_1,w)\in\Lambda}\Bigl\{|b-a|+\frac{c|w|^2}{n}-T(a,z)-T(z,b)\Bigr\}\leq\Os_1(CQ_n)\,.
\end{equation}
We cover~$\Lambda$ by boxes and bound~$T(a,z)$ and~$T(z,b)$ from below through points at distance~$n/16$ from the center of each box. For integers~$j\geq0$ and~$k\in\Z^{d-1}$, let~$t_j\coloneqq-C_0n+jn/500$ and
\begin{equation*}
B_{j,k}\coloneqq[t_j-n/1000,t_j+n/1000]\times(nQ_n)^{\nf12}\bigl(k+[-1,1]^{d-1}\bigr)\,;
\end{equation*}
we use only the boxes that meet~$\Lambda$, which cover it. Let~$x_\pm\coloneqq(t_j\pm n/16,(nQ_n)^{\nf12}k)$. For~$N$ large, Lemma~\ref{lem:tree} at the length~$n/16$ gives~$T(z,x_+)\leq e_1\cdot(x_+-z)+\Os_1(CQ_n)$ for~$z\in B_{j,k}$. By~\eqref{eq:parent} at the length~$|x_+-a|\in[C_0n,Cn]$, isotropy and the triangle inequality, for~$z\in B_{j,k}$,
\begin{equation*}
T(a,z)\geq e_1\cdot(z-a)+\frac{nQ_n|k|^2}{|x_+-a|+e_1\cdot(x_+-a)}-\Os_1(CQ_n)\geq e_1\cdot(z-a)+cQ_n|k|^2-\Os_1(CQ_n)\,.
\end{equation*}
With
\begin{equation*}
X_{j,k}\coloneqq\sup_{z\in B_{j,k}}\bigl(|b-a|+cQ_n|k|^2-T(a,z)-T(z,b)\bigr)_+\,,
\end{equation*}
this bound and the same bound for~$T(z,b)$, using~$x_-$, give~$X_{j,k}\leq\Os_1(CQ_n)$. Moreover, $|w|^2/n\leq Q_n(2|k|^2+C)$ for~$z=(z_1,w)\in B_{j,k}$, so
\begin{equation*}
|b-a|+\frac{c|w|^2}{4n}-T(a,z)-T(z,b)\leq X_{j,k}-\frac c2Q_n|k|^2+CQ_n\,.
\end{equation*} Summing the tails of~$X_{j,k}-cQ_n|k|^2/2$ over the boxes proves~\eqref{eq:boxpenalty}.

Finally, let~$Y\coloneqq\sup\{|w|:(z_1,w)\in\pi,\ \pi\in\Pi_{T_S}(u,v),\ u\in I,\ v\in I'\}$. If~$z$ lies on~$\pi\in\Pi_{T_S}(u,v)$, then~$T(a,z)+T(z,b)\leq T(a,u)+T(\pi)+T(v,b)$, so the two bounds give~$Y^2\leq\Os_1(CnQ_n)$, that is, $Y\leq\Os_2(C(nQ_n)^{\nf12})$. Since~$|y|\leq n/2$, the normal of~$S$ makes an angle at most~$\arctan\nf12$ with~$e_1$, so the points~$z=(z_1,w)$ of~$S$ satisfy~$-|w|/2\leq z_1\leq r+|w|/2$. Hence the points~$z$ of the curves satisfy~$\dist(z,[u,v])\leq2Y+2(d\,nQ_n)^{\nf12}$.
\end{proof}

\begin{proof}[Proof of Lemma~\ref{lem:modulus}]
We may again assume that the centers of~$I$ and~$I'$ are~$0$ and~$re_1$. Fix~$h\in(0,4^{-3})$, let~$\gamma\coloneqq h^{\nf13}$, let~$Z$ be the supremum of~$|T_S(u,v)-T_S(\widetilde u,\widetilde v)-|v-u|+|\widetilde v-\widetilde u||$ over the points in~\eqref{eq:modulus}, and let~$N$ be large, depending on~$K,h,\beta_1$. We bound~$Z$ separately on the event~$E$ that the left side of~\eqref{eq:localization} is at least~$1/\gamma$, and off~$E$. Then~$\Pp[E]\leq2e^{-c\gamma^{-2}}$ by~\eqref{eq:localization}, and~$\E[Z^{16d}]^{\nf1{16d}}\leq CQ_n$ by~\eqref{eq:profiles}. Hence the Cauchy--Schwarz inequality gives~$\E[Z^{8d}\mathbf 1_E]^{\nf1{8d}}\leq Ce^{-c\gamma^{-2}}Q_n$.

Off~$E$, we reroute the curves near their endpoints. With~$g$ as in~\eqref{eq:cubepair}, let~$S_1\coloneqq g([0,\gamma n]\times\R^{d-1})$ and~$S_2\coloneqq g([(1-\gamma)n,n]\times\R^{d-1})$, and let~$H_1\coloneqq g(\{\gamma n\}\times\R^{d-1})$ and~$H_2\coloneqq g(\{(1-\gamma)n\}\times\R^{d-1})$. Let~$u,\widetilde u\in I$ and~$v,\widetilde v\in I'$ with~$|u-\widetilde u|,|v-\widetilde v|\leq h(nQ_n)^{\nf12}$, let~$\pi\in\Pi_{T_S}(u,v)$, let~$x_1$ be its first point in~$H_1$, and let~$x_2$ be its last point in~$H_2$. Since~$\rho\leq\rho_S\leq\rho_{S_1}$ on~$S_1$, and similarly on~$S_2$, replacing the pieces of~$\pi$ before~$x_1$ and after~$x_2$ by curves in~$\Pi_{T_{S_1}}(\widetilde u,x_1)$ and~$\Pi_{T_{S_2}}(x_2,\widetilde v)$ gives
\begin{equation*}
T_S(\widetilde u,\widetilde v)-T_S(u,v)\leq T_{S_1}(\widetilde u,x_1)-T(u,x_1)+T_{S_2}(x_2,\widetilde v)-T(x_2,v)+3\,.
\end{equation*}

Let
\begin{equation*}
D\coloneqq\sup\bigl|T_P(z,z')-|z'-z|\bigr|\,,
\end{equation*}
where the supremum is over~$P\in\{\R^d,S_1\}$, $z\in I$ and~$z'\in H_1$, and over~$P\in\{\R^d,S_2\}$, $z\in I'$ and~$z'\in H_2$, with~$\dist(z',[0,re_1])\leq(\gamma^{-1}+2d^{\nf12})(nQ_n)^{\nf12}$. Off~$E$, the points~$x_1$ and~$x_2$ satisfy this condition. These pairs of points lie in at most~$C\gamma^{-4d}$ pairs of cubes as in~\eqref{eq:cubepair} with~$s=\gamma n$ and~$|y|\leq\gamma n$, and with~$S_1$ or~$S_2$ in place of~$S$. Hence~\eqref{eq:profiles} and the union bound~\eqref{eq:union} give~$D\leq\Os_1(CQ_{\gamma n}\log(2/\gamma))$. Writing~$|x|=e_1\cdot x+\phi(x)$, with~$\phi(x)\coloneqq|x|-e_1\cdot x$, the terms~$e_1\cdot x$ cancel, and
\begin{multline*}
T_S(\widetilde u,\widetilde v)-|\widetilde v-\widetilde u|-T_S(u,v)+|v-u|\\
\leq4D+3+\phi(x_1-\widetilde u)-\phi(x_1-u)+\phi(\widetilde v-x_2)-\phi(v-x_2)-\phi(\widetilde v-\widetilde u)+\phi(v-u)\,.
\end{multline*}
If~$e_1\cdot x\geq0$, then~$|\nabla\phi(x)|\leq2|x-(e_1\cdot x)e_1|/|x|$. This is at most~$C(1+\gamma^{-1})(nQ_n)^{\nf12}/(\gamma n)$ on the segments joining the arguments of the first two pairs of~$\phi$ terms, and at most~$C(Q_n/n)^{\nf12}$ on the segment joining those of the last pair. The arguments in each pair differ by at most~$2\gamma^3(nQ_n)^{\nf12}$, so the~$\phi$ terms add up to at most~$C\gamma Q_n$. Exchanging the roles of~$(u,v)$ and~$(\widetilde u,\widetilde v)$ gives~$Z\leq4D+3+C\gamma Q_n$ off~$E$. Since~$1\leq\gamma Q_n$ and~$Q_{\gamma n}\leq\gamma^{\nf25}Q_n$,
\begin{equation*}
\E[Z^{8d}]^{\nf1{8d}}\leq C\bigl(e^{-c\gamma^{-2}}+\gamma^{\nf25}\log(2/\gamma)+\gamma\bigr)Q_n\,.
\end{equation*}
Since~$\gamma=h^{\nf13}$, we have~$C(e^{-c\gamma^{-2}}+\gamma^{\nf25}\log(2/\gamma)+\gamma)\leq\eta$ if~$h$ is small, depending only on~$\eta$.
\end{proof}
\section{A logarithmic gain of variance}\label{sec:gain}

In this section we prove the bound~(b) of Section~\ref{ssec:overview}: $\Var[T_{Kn}]\leq\delta_KQ_{Kn}^2$ for~$Kn\in[N/8,N]$. We let~$\delta_K$ and~$\delta_L$ denote numbers depending only on~$K$, respectively on~$L$, that tend to zero as~$K\to\infty$, respectively~$L\to\infty$, and may change from line to line.

Let~$n$ be a length with~$Kn\in[N/8,N]$, and let~$F$, $T^{(i)}$, $I_i(j)$ and~$x_i$ be as defined before Lemma~\ref{lem:rowenergy}, with~$M=K$ and~$s=n$. For an integer~$L\geq1$ and vectors~$a_i\in\{0,\ldots,L-1\}^{d-1}$, we shift the blocks of Section~\ref{ssec:overview} in the~$i$th slab by~$(nQ_n)^{\nf12}a_i$, so that they are the sets~$[(i-1)n,in]\times(nQ_n)^{\nf12}(z+[0,L)^{d-1})$ for~$z\in a_i+L\Z^{d-1}$.

For each block~$G$, let~$\sigma_G$ be a random variable determined by~$\omega|_G$. The \emph{resampling} of~$G$ is the noise~$\omega^G$ that coincides with~$\omega$ outside~$G$ and whose restriction to~$G$ is drawn, given~$\omega$, from the conditional law of~$\omega|_G$ given~$\sigma_G$, independently for distinct blocks, so that~$(\omega^G,\omega)$ has the law of~$(\omega,\omega^G)$. With~$\sigma_G$ constant, $\omega^G|_G$ is an independent copy of~$\omega|_G$. Let~$T^{(i),G}$ and~$F^G$ be~$T^{(i)}$ and~$F$ computed with~$\omega^G$.

Changing~$\omega|_G$ can increase~$T^{(i)}(u,v)$ only if some curve in~$\Pi_{T^{(i)}}(u,v)$ passes within distance one of~$G$. For~$1\leq i\leq K$ and~$j,k\in\Z^{d-1}$ with~$(nQ_n)^{\nf12}|k-j|\leq n/2$, we therefore let
\begin{equation*}
R_i(j,k)\coloneqq(d-1)^{\nf12}+2+(nQ_n)^{-\nf12}\sup_{u\in I_{i-1}(j),\,v\in I_i(k)}\ \sup_{\pi\in\Pi_{T^{(i)}}(u,v)}\ \sup_{x\in\pi}\dist(x,[\bar u,\bar v])\,,
\end{equation*}
where~$\bar u$ and~$\bar v$ are the centers of~$I_{i-1}(j)$ and~$I_i(k)$, the \emph{tube}
\begin{equation*}
U_i(j,k)\coloneqq\bigl\{x:\dist(x,[\bar u,\bar v])\leq R_i(j,k)(nQ_n)^{\nf12}\bigr\}\,,
\end{equation*}
and
\begin{equation*}
Y_i(j,k)\coloneqq\frac1{Q_n^2}\sum_G\Bigl[\sup_{u\in I_{i-1}(j),\,v\in I_i(k)}\bigl(T^{(i),G}(u,v)-T^{(i)}(u,v)\bigr)_+\Bigr]^2\,,
\end{equation*}
where the sum is over the blocks~$G$ of the~$i$th slab. For the other~$j,k$, we let~$R_i(j,k)=Y_i(j,k)=0$. For~$N$ large, $nQ_n\geq1$, so~$U_i(j,k)$ contains the points within distance one of the curves in~$\Pi_{T^{(i)}}(u,v)$ for~$u\in I_{i-1}(j)$ and~$v\in I_i(k)$. The arrays~$Y_1,\ldots,Y_K$ are independent, and, by Assumption~\ref{ass:local}(1), the term of~$G$ in~$Y_i(j,k)$ vanishes unless~$G$ meets~$U_i(j,k)$.

The next lemma bounds the cost of resampling blocks along the minimizing chain, and the widths of the tubes.

\begin{lemma}\label{lem:blocks}
If $N$ is sufficiently large, depending only on~$K$, then
\begin{equation}\label{eq:radiusenergy}
\E\Bigl[\sum_{i=1}^KR_i(x_{i-1},x_i)^2\Bigr]\leq CK\,,
\end{equation}
and, for every integer~$L\geq1$, every choice of the vectors~$a_i$ and of the random variables~$\sigma_G$, every~$i\in\{1,\ldots,K\}$ and every pair $j,k$ with $(nQ_n)^{\nf12}|k-j|\leq n/2$,
\begin{equation}\label{eq:rewardmoment}
\E[Y_i(j,k)^{16d}]^{\nf1{16d}}\leq C(1+|j-k|)\qquad\text{and}\qquad
\sum_G\E[(F^G-F)^2]\leq CKQ_n^2\,,
\end{equation}
where the sum is over the blocks of all~$K$ slabs.
\end{lemma}

\begin{proof}
Lemma~\ref{lem:localization} at the length~$n\geq N/(8K)$ gives~$R_i(j,k)\leq\Os_2(C)$. The arrays~$R_i^2$ are independent over~$i$, so~\eqref{eq:chainsum}, applied to a small multiple of~$R_i^2$, gives~\eqref{eq:radiusenergy}.

For the first bound in~\eqref{eq:rewardmoment}, we cover the~$i$th slab by the \emph{cells}
\begin{equation*}
[(i-1)n,in]\times(nQ_n)^{\nf12}\bigl(z+[0,1)^{d-1}\bigr)\,,\qquad z\in\Z^{d-1}\,.
\end{equation*}
The term of a block~$G$ in~$Y_i(j,k)$ is at most the number of cells of~$G$ that meet~$U_i(j,k)$, times~$Q_n^{-2}\sup_{u,v}|T^{(i),G}(u,v)-T^{(i)}(u,v)|^2$. Since~$\omega^G$ has the law of~$\omega$, every moment of~$Q_n^{-1}\sup_{u,v}|T^{(i),G}(u,v)-T^{(i)}(u,v)|$ is bounded by~\eqref{eq:profiles}, uniformly in~$G$, in~$L$ and in the~$\sigma_G$. The cell of~$z$ meets~$U_i(j,k)$ only if~$R_i(j,k)\geq\dist(z,[j,k])-C$, where~$[j,k]\subseteq\R^{d-1}$ is the segment between the indices. Hence H\"older's inequality and~$R_i(j,k)\leq\Os_2(C)$ bound the~$16d$th moment of the term of the cell of~$z$ in this bound by~$Ce^{-c\dist(z,[j,k])^2}$. Summing over~$z$, by Minkowski's inequality, gives the bound.

For the second bound, evaluating~$F^G$ at the points~$x_i$ gives, for~$N$ large,
\begin{equation*}
\sum_G(F^G-F)_+^2\leq Q_n^2\sum_{i=1}^KY_i(x_{i-1},x_i)\,.
\end{equation*} By exchangeability, $(F-F^G)_+$ has the law of~$(F^G-F)_+$, so
\begin{equation}\label{eq:blocksum}
\sum_G\E\bigl[(F^G-F)^2\bigr]\leq2Q_n^2\,\E\Bigl[\sum_{i=1}^KY_i(x_{i-1},x_i)\Bigr]\,.
\end{equation}
For~$|j-k|\geq1$, the first bound in~\eqref{eq:rewardmoment} gives
\begin{equation*}
\E\bigl[(Y_i(j,k)-|j-k|^2)_+^{4d}\bigr]^{\nf1{4d}}\leq|j-k|^{-6}\E\bigl[Y_i(j,k)^{16d}\bigr]^{\nf1{4d}}\leq C|j-k|^{-2}\,.
\end{equation*}
Hence~\eqref{eq:chainsum}, applied to a small multiple of the arrays~$Y_i$, gives~$\E[\sum_iY_i(x_{i-1},x_i)]\leq CK$, and~\eqref{eq:blocksum} gives the second bound.
\end{proof}

To apply Proposition~\ref{prop:BR}, we approximate the normalized passage times~$(T^{(i)}(u,v)-n)/Q_n$ across each block~$G$ by a function of~$\omega|_G$ with finitely many values.

\begin{lemma}\label{lem:labels}
Let~$L\geq16$ be an integer. There exists~$\alpha_L>0$, depending only on~$L$, such that the following holds if~$N$ is sufficiently large, depending only on~$K$ and~$L$. For every choice of the vectors~$a_i$ and every~$i\in\{1,\ldots,K\}$, there exist random variables~$X_G$, one for each block~$G$ of the~$i$th slab, such that~$X_G$ is determined by~$\omega|_G$, takes finitely many values, each a bounded function of~$(u,v)$ and each with probability at least~$\alpha_L$, and satisfies
\begin{equation}\label{eq:labelprofile}
\E\Bigl[\sup_{u,v}\Bigl|\frac{T^{(i)}(u,v)-n}{Q_n}-X_G(u,v)\Bigr|^{8d}\Bigr]^{\nf1{8d}}\leq\delta_L\,,
\end{equation}
where~$u$ and~$v$ range over the points of the two faces of~$G$ at distance at least~$(LnQ_n)^{\nf12}$ from the other blocks of the~$i$th slab. Moreover, for every block~$G$ of the~$i$th slab and every~$\tau$ orthogonal to~$e_1$ such that~$G+\tau$ is a block, $X_{G+\tau}(u+\tau,v+\tau)$, computed with~$\omega$ translated by~$\tau$, equals~$X_G(u,v)$.
\end{lemma}

\begin{proof}
\emph{Step 1.} We compare~$T^{(i)}$ with~$T_G$, which is defined by~\eqref{eq:slabtime} for the closure of~$G$ and determined by~$\omega|_G$. On~$G$, we have~$\rho_G\geq\rho^{(i)}$, with equality at distance more than one from the other blocks of the~$i$th slab. Hence~$T_G(u,v)\geq T^{(i)}(u,v)$ for~$u,v\in G$, with equality if the curves in~$\Pi_{T^{(i)}}(u,v)$ stay at distance more than one from the other blocks. Let~$E$ be the event that, for some~$u$ and~$v$ on the faces of~$G$ at distance at least~$(LnQ_n)^{\nf12}/2$ from the other blocks, some curve in~$\Pi_{T^{(i)}}(u,v)$ comes within distance one of another block. By~\eqref{eq:localization} for the at most~$CL^{2(d-1)}$ pairs of cubes~$I_{i-1}(j)$ and~$I_i(k)$ covering these points, $\Pp[E]\leq CL^{2(d-1)}e^{-cL}$.

\emph{Step 2.} We discretize~$T_G$. Let~$G=[(i-1)n,in]\times(nQ_n)^{\nf12}(z+[0,L)^{d-1})$ and~$\kappa\in(0,(d-1)^{-\nf12}]$. For~$u=((i-1)n,(nQ_n)^{\nf12}(z+w))$ and~$v=(in,(nQ_n)^{\nf12}(z+w'))$, let~$u_\kappa$ and~$v_\kappa$ be~$u$ and~$v$ with~$w$ and~$w'$ replaced by~$\kappa\lfloor w/\kappa\rfloor$ and~$\kappa\lfloor w'/\kappa\rfloor$, coordinatewise, and
\begin{equation*}
\Psi_G(u,v)\coloneqq\kappa\Bigl\lfloor\frac1\kappa\max\Bigl\{-L^4,\min\Bigl\{L^4,\frac{T_G(u_\kappa,v_\kappa)-n}{Q_n}\Bigr\}\Bigr\}\Bigr\rfloor\,.
\end{equation*}
Then~$\Psi_G$ is determined by~$\omega|_G$. As~$(u_\kappa,v_\kappa)$ takes at most~$\lceil L/\kappa\rceil^{2(d-1)}$ values, and~$\Psi_G(u,v)$ lies in~$\kappa\Z\cap[-L^4-\kappa,L^4]$, the function~$\Psi_G$ takes at most~$J\coloneqq(2L^4/\kappa+2)^{\lceil L/\kappa\rceil^{2(d-1)}}$ values. For~$u$ and~$v$ as in~\eqref{eq:labelprofile}, the points~$u_\kappa$ and~$v_\kappa$ are at distance at least~$(LnQ_n)^{\nf12}/2$ from the other blocks. Hence Step~1 gives
\begin{multline*}
|T^{(i)}(u,v)-n-Q_n\Psi_G(u,v)|\leq|T^{(i)}(u,v)-T^{(i)}(u_\kappa,v_\kappa)|\\
+\bigl(|T^{(i)}(u_\kappa,v_\kappa)-n|-L^4Q_n\bigr)_++2L^4Q_n\mathbf 1_E+\kappa Q_n\,.
\end{multline*}
The cubes covering the faces of~$G$, enlarged to side~$4(nQ_n)^{\nf12}$, form at most~$CL^{2(d-1)}$ pairs as in~\eqref{eq:cubepair} with~$s=n$ and~$|y|\leq CL(nQ_n)^{\nf12}\leq n/2$. On these pairs, we have~$\bigl||v-u|-n\bigr|\leq CL^2Q_n$, and~$|v-u|$ changes by at most~$CL\kappa Q_n$ when~$u$ and~$v$ are replaced by~$u_\kappa$ and~$v_\kappa$. Hence~\eqref{eq:modulus}, applied to each of these pairs, gives~$\E[\sup_{u,v}|T^{(i)}(u,v)-T^{(i)}(u_\kappa,v_\kappa)|^{8d}]\leq(\delta_LQ_n)^{8d}$ if~$\kappa$ is small, depending on~$L$. By~\eqref{eq:profiles} and~\eqref{eq:union}, $\sup_{u,v}|T^{(i)}(u_\kappa,v_\kappa)-n|\leq CL^2Q_n+\Os_1(CQ_n\log(2+L))$, so
\begin{equation*}
\E\Bigl[\sup_{u,v}\bigl(|T^{(i)}(u_\kappa,v_\kappa)-n|-L^4Q_n\bigr)_+^{8d}\Bigr]\leq(\delta_LQ_n)^{8d}\,.
\end{equation*}
Also, $\E[(2L^4Q_n\mathbf 1_E)^{8d}]=(2L^4Q_n)^{8d}\Pp[E]\leq(\delta_LQ_n)^{8d}$. Hence~\eqref{eq:labelprofile} holds with~$\Psi_G$ in place of~$X_G$.

\emph{Step 3.} We merge the values of small probability. We write~$\Psi_G$ as a function of the coordinates~$(w,w')$ of Step~2. For every~$\tau$ orthogonal to~$e_1$ such that~$G+\tau$ is a block, $\Psi_{G+\tau}$, computed with~$\omega$ translated by~$\tau$, is then the same function as~$\Psi_G$. Hence, by stationarity, $p(\psi)\coloneqq\Pp[\Psi_G=\psi]$ does not depend on the block~$G$ of the~$i$th slab. Fix~$\psi^*$ with~$p(\psi^*)=\max p$, and let~$X_G\coloneqq\Psi_G$ if~$p(\Psi_G)\geq\alpha_L$ and~$X_G\coloneqq\psi^*$ otherwise. The~$X_G$ have the translation property of the lemma. Since~$\Psi_G$ takes at most~$J$ values, $p(\psi^*)\geq1/J$, so every value of~$X_G$ has probability at least~$\alpha_L$ if~$\alpha_L\leq1/J$. Since~$|X_G-\Psi_G|\leq3L^4$, and~$X_G=\Psi_G$ off the event~$\{p(\Psi_G)<\alpha_L\}$, which has probability at most~$\alpha_LJ$, the replacement adds at most~$3L^4(\alpha_LJ)^{\nf1{8d}}$ to the left side of~\eqref{eq:labelprofile}. This is at most~$1/L$ if~$\alpha_L$ is small, depending on~$L$ and~$\kappa$, and~$\kappa$ depends only on~$L$.
\end{proof}

From now on, the~$X_G$, for the blocks~$G$ of all~$K$ slabs, are random variables as in Lemma~\ref{lem:labels}. Since~$X_G$ is determined by~$\omega|_G$ and distinct blocks have disjoint interiors, the~$X_G$ of distinct blocks are independent. We average over translations: for~$0\leq m<\lceil K^{\nf23}\rceil$, let~$F_m$ be~$F_\tau$ of Section~\ref{sec:prelim} with~$\tau=mL(nQ_n)^{\nf12}e_2$, let~$F_m^G$ be~$F_m$ computed with~$\omega^G$, and let~$\overline F\coloneqq\lceil K^{\nf23}\rceil^{-1}\sum_mF_m$. With~$f\coloneqq\E[\overline F\mid(X_G)_G]$,
\begin{equation*}
\Var[\overline F]=\E\bigl[\Var[\overline F\mid(X_G)_G]\bigr]+\Var[f]\,.
\end{equation*}
The next lemma bounds the first term. We bound the second by Proposition~\ref{prop:BR}, with the~$X_G$ as the variables and~$\E_G[f]$ the expectation of~$f$ with respect to~$X_G$ alone.

\begin{lemma}\label{lem:residual}
Let~$L\geq16$ be an integer, and let~$\omega^G$ be the resampling of~$G$ with~$\sigma_G=X_G$. If~$N$ is sufficiently large, depending only on~$K$ and~$L$, then
\begin{equation}\label{eq:residual}
\frac1{L^{K(d-1)}}\sum_a\sum_G\E\bigl[(F^G-F)^2\bigr]\leq\delta_LKQ_n^2\qquad\text{and}\qquad\frac1{L^{K(d-1)}}\sum_a\E\bigl[\Var[\overline F\mid(X_G)_G]\bigr]\leq\delta_LKQ_n^2\,,
\end{equation}
where~$a=(a_1,\ldots,a_K)$ ranges over~$(\{0,\ldots,L-1\}^{d-1})^K$; the blocks, the~$X_G$ and~$\omega^G$ depend on~$a$.
\end{lemma}

\begin{proof}
We begin with the first bound in~\eqref{eq:residual}, with the~$a_i$ independent of the rest and included in probabilities and expectations. Let~$Y_i(j,k)$ be defined with this resampling, and let~$\eta\in(0,1]$. We show that~$\E[(Y_i(j,k)-\eta|j-k|^2)_+^{4d}]^{\nf1{4d}}\leq\eta$ for all~$i,j,k$ if~$L$ is large depending on~$\eta$. Then~\eqref{eq:chainsum}, applied to the arrays~$Y_i/\eta$, gives~$\E[\sum_iY_i(x_{i-1},x_i)]\leq C\eta K$, and~\eqref{eq:blocksum} gives the first bound in~\eqref{eq:residual}.

If~$|j-k|\geq1$, then~\eqref{eq:rewardmoment} gives~$\E[(Y_i(j,k)-\eta|j-k|^2)_+^{4d}]^{\nf1{4d}}\leq C\eta^{-3}|j-k|^{-2}$, which is at most~$\eta$ if~$|j-k|$ is large depending on~$\eta$. It remains to show that~$\E[Y_i(j,k)^{4d}]$ tends to zero as~$L\to\infty$ for each fixed value of~$|j-k|$, uniformly in~$i$. Let~$G$ be the block containing the center~$\bar u$ of~$I_{i-1}(j)$. Let~$E$ be the event that~$\bar u$ lies within distance~$C(1+|j-k|+L^{\nf12})(nQ_n)^{\nf12}$ of another block or that~$R_i(j,k)>L^{\nf12}$. Since~$a_i$ is uniform, \eqref{eq:localization} gives~$\Pp[E]\leq C(1+|j-k|+L^{\nf12})/L+Ce^{-cL}$. Off~$E$, only~$G$ contributes to~$Y_i(j,k)$, and the cubes~$I_{i-1}(j)$ and~$I_i(k)$ lie in the region of~\eqref{eq:labelprofile}. The noise~$\omega^G$ has the law of~$\omega$, and~$X_G$ takes the same value at~$\omega$ and~$\omega^G$, so, off~$E$,
\begin{multline*}
\frac1{Q_n}\sup_{u,v}\bigl(T^{(i),G}(u,v)-T^{(i)}(u,v)\bigr)_+\\
\leq\sup_{u,v}\Bigl|\frac{T^{(i),G}(u,v)-n}{Q_n}-X_G(u,v)\Bigr|+\sup_{u,v}\Bigl|\frac{T^{(i)}(u,v)-n}{Q_n}-X_G(u,v)\Bigr|\,,
\end{multline*}
and~\eqref{eq:labelprofile} bounds both terms. On~$E$, we use H\"older's inequality and~\eqref{eq:rewardmoment}.

For the second bound, given~$(X_G)_G$, the~$\omega|_G$ are conditionally independent, and~$\omega^G|_G$ is a conditionally independent copy of~$\omega|_G$. The Efron--Stein inequality~\cite[Proposition~1]{BLM03}, extended to countably many independent variables by martingale convergence, applied conditionally to~$\overline F$, with Jensen's inequality over~$m$, gives
\begin{equation*}
\E\bigl[\Var[\overline F\mid(X_G)_G]\bigr]\leq\frac1{2\lceil K^{\nf23}\rceil}\sum_{m,G}\E\bigl[(F_m^G-F_m)^2\bigr]\,.
\end{equation*}
Translations by multiples of~$L(nQ_n)^{\nf12}e_2$ map blocks to blocks and, by Lemma~\ref{lem:labels}, the~$X_G$ of a block to that of the translated block. Hence, by stationarity, the sum over~$G$ is the same for every~$m$, and the first bound in~\eqref{eq:residual} bounds the right side.
\end{proof}

The next lemma bounds the two sums in Proposition~\ref{prop:BR}.

\begin{lemma}\label{lem:fewtranslates}
Let~$L\geq16$ be an integer. If~$N$ is sufficiently large, depending only on~$K$ and~$L$, then, for every choice of the vectors~$a_i$,
\begin{equation*}
\sum_G\E\bigl[(f-\E_G[f])^2\bigr]\leq CKQ_n^2\qquad\text{and}\qquad\sum_G\E\bigl[|f-\E_G[f]|\bigr]^2\leq CK^{-\nf16}KQ_n^2\,.
\end{equation*}
\end{lemma}

\begin{proof}
Let~$\omega^G$ be the resampling of~$G$ with~$\sigma_G$ constant, so that
\begin{equation*}
f-\E_G[f]=\frac1{\lceil K^{\nf23}\rceil}\sum_m\E\bigl[F_m-F_m^G\bigm|(X_H)_H\bigr]\,.
\end{equation*} Let~$G$ lie in the~$i$th slab, and let~$x^m$ and~$\widetilde x^m$ be the lexicographically least chains attaining~$F_m$ and~$F_m^G$. Evaluating~$F_m^G$ at~$x^m$ and~$F_m$ at~$\widetilde x^m$ shows that~$F_m^G=F_m$ unless~$G$ meets~$U_i(x_{i-1}^m,x_i^m)$, or~$U_i(\widetilde x_{i-1}^m,\widetilde x_i^m)$ computed with~$\omega^G$.

The tube~$U_i(x_{i-1},x_i)$ meets at most
\begin{equation*}
C\Bigl(1+\frac{|x_i-x_{i-1}-ne_1|}{(nQ_n)^{\nf12}}+R_i(x_{i-1},x_i)\Bigr)
\end{equation*}
translates of~$G$ by multiples of~$L(nQ_n)^{\nf12}e_2$. By~\eqref{eq:rowenergy} and~\eqref{eq:radiusenergy}, the squares of these numbers, summed over~$i$, have mean at most~$CK$, so, by the Cauchy--Schwarz inequality, each has mean at most~$CK^{\nf12}$. The translation by~$-mL(nQ_n)^{\nf12}e_2$ maps~$F_m$ to~$F$ and~$G$ to such a translate. Hence, by stationarity, for an independent~$m$ uniform in~$\{0,\ldots,\lceil K^{\nf23}\rceil-1\}$, the block~$G$ meets either tube of~$F_m$ with probability at most~$2CK^{\nf12}/\lceil K^{\nf23}\rceil\leq CK^{-\nf16}$ (Figure~\ref{fig:averaging}). Jensen's inequality, and the Cauchy--Schwarz inequality for the joint law of~$m$ and~$\omega$, give
\begin{equation*}
\E\bigl[(f-\E_G[f])^2\bigr]\leq\sum_m\frac{\E[(F_m-F_m^G)^2]}{\lceil K^{\nf23}\rceil}\qquad\text{and}\qquad\E\bigl[|f-\E_G[f]|\bigr]^2\leq CK^{-\nf16}\sum_m\frac{\E[(F_m-F_m^G)^2]}{\lceil K^{\nf23}\rceil}\,.
\end{equation*}
By stationarity, \eqref{eq:rewardmoment} holds also for~$F_m$, and summing over~$G$ proves both bounds.
\end{proof}

\begin{proposition}[The gain]\label{prop:gain}
If~$K$ is sufficiently large and~$N$ is sufficiently large depending on~$K$, then
\begin{equation*}
\Var[T_{Kn}]\leq\delta_KQ_{Kn}^2\qquad\text{for every~$n$ with }Kn\in[N/8,N]\,.
\end{equation*}
\end{proposition}

\begin{proof}
Let~$n$ satisfy~$Kn\in[N/8,N]$, and let~$L\geq16$ be an integer, to be chosen below depending only on~$K$. All the bounds below hold for~$N$ large depending on~$K$ and~$L$, uniformly in~$n$ and~$A$.

We first bound~$\Var[f]$, for every choice of the vectors~$a_i$. Proposition~\ref{prop:BR} with~$\alpha=\alpha_L$, which applies by Lemma~\ref{lem:labels}, and Lemma~\ref{lem:fewtranslates} give
\begin{equation*}
\Var[f]\log\Bigl(\frac{\Var[f]}{CK^{-\nf16}KQ_n^2}\Bigr)\leq\frac{CKQ_n^2}{\alpha_L}\,.
\end{equation*}
If~$\Var[f]\geq K^{-\nf1{12}}KQ_n^2$, then the logarithm is at least~$\log(K^{\nf1{12}}/C)$, which is at least~$(\log K)/24$ for~$K$ large. Hence~$\Var[f]\leq(K^{-\nf1{12}}+C/(\alpha_L\log K))KQ_n^2$.

Since~$\overline F$ does not depend on the vectors~$a_i$, we may average the decomposition of~$\Var[\overline F]$ over them. Lemma~\ref{lem:residual} and the bound on~$\Var[f]$ then give~$\Var[\overline F]\leq(\delta_L+K^{-\nf1{12}}+C/(\alpha_L\log K))KQ_n^2$.

We now let~$L$ grow slowly with~$K$. The error~$\delta_L$ needs~$L\to\infty$, the term~$C/(\alpha_L\log K)$ needs~$\alpha_L\log K\to\infty$, and the comparison of~$\overline F$ with~$T_{Kn}$ below needs~$L\leq K^{\nf1{24}}$. We choose~$L$ as the largest integer in~$[16,K^{\nf1{24}}]$ with~$\alpha_L\geq(\log K)^{-\nf12}$. It tends to infinity with~$K$, so~$\Var[\overline F]\leq\delta_KKQ_n^2$.

It remains to compare~$\overline F$ with~$T_{Kn}$. For the slab~$S\coloneqq[0,Kn]\times\R^{d-1}$, we have~$\rho^{(i)}\geq\rho_S$ on the~$i$th slab, so concatenating curves gives~$T(\tau,\tau+Kne_1)\leq T_S(\tau,\tau+Kne_1)\leq F_\tau$. Since~$(nQ_n/(KnQ_{Kn}))^{\nf12}=K^{-\nf34+1/(2K)}$ and~$L\leq K^{\nf1{24}}$, the vectors~$\tau$ of the translates satisfy~$|\tau|\leq\lceil K^{\nf23}\rceil L(nQ_n)^{\nf12}\leq2K^{-\nf1{24}+1/(2K)}(KnQ_{Kn})^{\nf12}$. Hence~\eqref{eq:modulus} at the length~$Kn$ with this slab~$S$, the bound of Lemma~\ref{lem:slab} on~$F_\tau-T(\tau,\tau+Kne_1)$, and Jensen's inequality give~$\E[(\overline F-T_{Kn})^2]^{\nf12}\leq\delta_KQ_{Kn}$. Since~$KQ_n^2=K^{\nf2K}Q_{Kn}^2$ and~$K^{\nf1K}\leq2$, these bounds give~$\Var[T_{Kn}]^{\nf12}\leq\delta_KQ_{Kn}$.
\end{proof}
\section{The mean and the induction step}\label{sec:completion}

In this section we show that~$\E[T_s]-s\leq Q_s/100$ when~$\Var[T_r/Q_r]$ is small for~$r\in[s,2s]$, and we complete the proof of the induction step, Proposition~\ref{prop:induction}.

\begin{lemma}[Nonrandom fluctuations]\label{lem:smallmean}
If~$N$ is sufficiently large, depending only on~$K$, then~$\E[T_s]-s\leq Q_s/100$ for every~$s\in[N/8,N/4]$ such that~$\sup_{s\leq r\leq2s}\Var[T_r/Q_r]$ is at most~$c$.
\end{lemma}

\begin{proof}
Let~$F,T^{(i)},I_i(j)$ and~$x_i$ be as defined before Lemma~\ref{lem:rowenergy}, for~$M=\lceil s\rceil$ slabs of width~$s$. For~$1\leq i\leq\lceil s\rceil$ and~$j,k\in\Z^{d-1}$, let
\begin{equation*}
X_i(j,k)\coloneqq\frac1{Q_s}\sup_{u\in I_{i-1}(j),\,v\in I_i(k)}|T^{(i)}(u,v)-\E[T^{(i)}(u,v)]|\qquad\text{if }(sQ_s)^{\nf12}|k-j|\leq s/2\,,
\end{equation*}
and~$X_i(j,k)\coloneqq0$ otherwise. The arrays~$X_i$ are independent over~$i$.

We first show that, for every~$\kappa\in(0,1]$, we have~$\E[X_i(j,k)^{4d}]^{\nf1{4d}}\leq\kappa$ if~$\sup_{s\leq r\leq2s}\Var[T_r/Q_r]$ is small and~$N$ is large, depending on~$\kappa$. For~$u\in I_{i-1}(j)$ and~$v\in I_i(k)$ with~$(sQ_s)^{\nf12}|k-j|\leq s/2$, we have~$s\leq|v-u|\leq2s$ for~$N$ large, and~$T(u,v)$ has the law of~$T_{|v-u|}$. By the hypothesis on the variance, \eqref{eq:slab} and~\eqref{eq:profiles}, $D\coloneqq(T^{(i)}(u,v)-\E[T^{(i)}(u,v)])/Q_s$ satisfies~$\E[D^{8d}]\leq C$ and~$\E[D^2]\leq C\sup_{s\leq r\leq2s}\Var[T_r/Q_r]+C\log^4(2+s)/Q_s^2$, so, by H\"older's inequality, $\E[D^{4d}]\leq C\E[D^2]^c$. For a fixed~$h\in(0,1]$, the same holds, with~$C$ depending also on~$h$, for the maximum of~$|D|$ over a mesh of~$I_{i-1}(j)\times I_i(k)$ of spacing~$h(sQ_s)^{\nf12}$. Subtracting~$|v-u|$ from both~$T^{(i)}(u,v)$ and its mean changes nothing, so~\eqref{eq:modulus} extends this to the cubes. Choosing~$h$ small, then the variance small and then~$N$ large, all depending on~$\kappa$, proves the claim.

For~$r\geq s$, we have~$\E[T_r]\geq r$ and, by subadditivity, $\E[T_r]\geq\E[T_s]-(r-s)\sup\rho$. Averaging these two bounds with the weights~$\sup\rho/(1+\sup\rho)$ and~$1/(1+\sup\rho)$ gives~$\E[T_r]\geq s+(\E[T_s]-s)/(1+\sup\rho)$. For~$u$ and~$v$ as above, since~$T^{(i)}\geq T$, this bound at~$r=|v-u|$ gives
\begin{equation*}
T^{(i)}(u,v)-s\geq\frac{\E[T_s]-s}{1+\sup\rho}-Q_sX_i(j,k)\,.
\end{equation*} By~\eqref{eq:chainsum}, applied to the arrays~$X_i/\kappa$, we have~$\E[\sum_iX_i(x_{i-1},x_i)]\leq C\kappa\lceil s\rceil$. Summing the bound on~$T^{(i)}(u,v)-s$ at~$u=x_{i-1}$ and~$v=x_i$ therefore gives
\begin{equation*}
\E[F]-s\lceil s\rceil\geq\lceil s\rceil\Bigl(\frac{\E[T_s]-s}{1+\sup\rho}-C\kappa Q_s\Bigr)\,.
\end{equation*} By~\eqref{eq:columns} and~\eqref{eq:classical} at the length~$s\lceil s\rceil\leq2s^2$, the left side is at most~$C\lceil s\rceil\log^2(2+s)$. Hence~$\E[T_s]-s\leq C\kappa Q_s+C\log^2(2+s)$. Since~$Q_s\geq s^{\nf12-\nf1K}$, this is at most~$Q_s/100$ if~$\kappa$ is small and~$N$ is large.
\end{proof}

\begin{proof}[Proof of Proposition~\ref{prop:induction}]
Let~$t\in[N/8,N]$. By~\eqref{eq:parent}, $\E[\exp(8|T_t-\E[T_t]|/Q_t)]\leq C$. Hence the Cauchy--Schwarz inequality gives
\begin{equation*}
\E\bigl[\exp\bigl(4|T_t-\E[T_t]|/Q_t\bigr)\bigr]\leq1+C\Var[T_t/Q_t]^{\nf12}\,.
\end{equation*} By Proposition~\ref{prop:gain} with~$n=t/K$, we may fix~$K$ so that, for~$N$ large depending on~$K$, $C\Var[T_t/Q_t]^{\nf12}\leq1$ and~$\Var[T_t/Q_t]$ is as small as Lemma~\ref{lem:smallmean} requires. With~$t=N$, this gives~$|T_N-\E[T_N]|\leq\Os_1(Q_N/4)$. Lemma~\ref{lem:smallmean} at~$s=N/4$ and subadditivity give
\begin{equation*}
0\leq\E[T_N]-N\leq4\bigl(\E[T_{N/4}]-N/4\bigr)\leq Q_N/25\,.
\end{equation*}
Hence~\eqref{eq:improvement} follows by~\eqref{eq:sum}, and all the thresholds on~$N$ depend only on~$K$.
\end{proof}

\end{document}